\documentclass[reqno]{amsart}
\usepackage{amsmath,amssymb,cite}
\usepackage[mathscr]{euscript}
\usepackage{mathtools}
\mathtoolsset{showonlyrefs}
\usepackage{xcolor}
\usepackage{braket}

\usepackage{amsmath,stackrel}
\usepackage{tikz}
\usetikzlibrary{shapes.geometric, arrows}
\tikzstyle{startstop} = [rectangle, rounded corners, minimum width=0.5cm, minimum height=0.5cm,text centered, draw=black]
\tikzstyle{io} = [rectangle, minimum width=0.5cm, minimum height=0.5cm, text centered, draw=black]
\tikzstyle{decision} = [diamond, minimum width=0.5cm, minimum height=0.5cm, text centered, draw=black]
\tikzstyle{arrow} = [thick,->,>=stealth]
\usepackage{ulem}

\definecolor{bblue}{rgb}{.2,0.2,.8}

\theoremstyle{plain}
\newtheorem{theorem}{Theorem}[section]
\newtheorem{proposition}[theorem]{Proposition}
\newtheorem{lemma}[theorem]{Lemma}

\theoremstyle{definition}

\theoremstyle{remark}
\newtheorem{remark}[theorem]{Remark}

\numberwithin{equation}{section}
\numberwithin{theorem}{section}

\def\be{\begin{equation}}
	\def\ee{\end{equation}}
\def\bp{\begin{pmatrix}}
	\def\ep{\end{pmatrix}}
\def\bea{\begin{eqnarray}}
	\def\eea{\end{eqnarray}}

\def\\{\par\medskip}

\let\0=\noindent

\newcommand{\abscont}{%
	\mathrel{<\mkern-15mu\raisebox{0.29ex}{$\scriptscriptstyle<$}}}

\newcommand{\id}{{1 \mskip -5mu {\rm I}}}
\renewcommand{\epsilon}{\varepsilon}
\renewcommand{\hat}{\widehat}

\begin{document}
	
	\title[LD and MPA]{Large deviations and the matrix product ansatz}

	\author{Davide Gabrielli}
	\address{\noindent Davide Gabrielli \hfill\break\indent 
		DISIM, Universit\`a dell'Aquila
		\hfill\break\indent 
		67100 Coppito, L'Aquila, Italy
	}
	\email{davide.gabrielli@univaq.it}

	\author{Federica Iacovissi}
	\address{\noindent Federica Iacovissi \hfill\break\indent 
		DISIM, Universit\`a dell'Aquila
		\hfill\break\indent 
		67100 Coppito, L'Aquila, Italy
	}
	\email{federica.iacovissi@graduate.univaq.it}

	\begin{abstract}
		We consider probability measures on $A^N$, the set of sequences of symbols on a finite alphabet $A$ of length $N$, that give a weight to each sequence in terms of a collection of matrices with non-negative entries and having rows and columns labeled by a finite or countable set $B$. We prove for such kind of measures large deviations principles for several empirical measures. Our approach is based on a simultaneous combination of an enlargement of the state space to sequences on $A\times B$ and a spectral conjugation that produces a stochastic matrix, as discussed in \cite{GI1}. As a result we describe the measures as hidden Markov measures and can deduce the large deviations results by contraction from the corresponding ones for the enlarged Markov chain. The measure on the enlarged state space is a Markov bridge. The invariant measures of several non equilibrium models of interacting particle systems can be represented by the so called {\it Matrix Product Ansatz} that corresponds to measures of the type that we consider and with matrices labeled by $B$ that is typically countable infinite. The large deviations behavior is different in the cases with $B$ finite or countable. In the finite case we give a variational formula for both the algebraic and the spatial empirical measures, that can be solved in special cases. For the infinite case, we illustrate the method through an example that is the invariant measure of the boundary driven TASEP model in a special regime. We recover in this way the celebrated results in \cite{Der4,Derr7}, and in particular we obtain a variational representation of the rate function similar to that in \cite{Bryc}. Our approach is general and can in principle be applied to any measure represented by the matrix product ansatz with matrices having positive entries.
	\end{abstract}
	
	\noindent
	
	
	\maketitle
	\thispagestyle{empty}
	
	\section{Introduction}
	
	We consider a class of probability measures that are called {\it Rational Models} in the area of theoretical informatics, in particular for the
	analysis of pattern statistics (see, for example, \cite{B,G2} and references therein). A rational model is a probability measure on symbolic sequences defined by using a collection of matrices with non-negative elements and two vectors with non-negative coordinates. An informal definition of the measure is the following. Consider a finite alphabet and associate a matrix with each element. To a symbolic sequence we associate the matrix obtained by the row by column product of all the matrices associated with the symbols, in the order that they appear. Finally, one computes the bilinear form of this matrix between the two fixed vectors. The resulting positive number is the weight associated with the symbolic sequence. Normalizing to one we obtain a probability measure that we call a rational model.
	
	\smallskip
	
	We prove some large deviations principles for rational models. In particular, we will introduce several notions of {\it Empirical Measures} both algebraically, defined in terms of frequencies, and spatially defined by considering the sequence rescaled and embedded on a unit interval. This second definition is natural considering the symbolic sequence as a configuration of particles and having in mind a scaling limit. The phenomenology of this class of models can be very different depending on several features of the matrices and of the vectors. We will concentrate on the major distinction when the matrices are finite or infinite. For the finite case we will just consider the simpler case of irreducible and aperiodic matrices, while for the infinite case we will consider a remarkable example.
	
	\smallskip
	We will see that, while in the finite case (at least under the assumptions on the matrices) we have a general compact variational representation for the rate functional of both algebraic and spatial empirical measures, in the infinite case we may have different behaviors. This is the counterpart of the fact that irreducible finite Markov chains are recurrent while in the infinite case different behaviors may occur. The specific infinite example that we discuss has not a joint large deviations principle for the algebraic empirical measures while that for the spatial empirical measure has a non local structure.
	
	\smallskip
	
	In non-equilibrium statistical mechanics a very important role is played by the {\it Matrix Product Ansatz} (see \cite{reviewhold,review} for some reviews) that allows to give combinatorial representations of stationary non-equilibrium states. The matrix product ansatz represents the invariant measure of some non reversible particle systems as a rational model, with the matrices and the vectors satisfying suitable algebraic relations. Typically, there are several representations of matrices satisfying such relations, which are infinite dimensional and often have positive entries, so that the invariant measure is exactly a rational model according to our definition. We will consider a special case of the boundary driven TASEP recovering with a different approach the large deviations results in \cite{Der5,Derr7,Bryc}, but our approach can be applied in principle to any probability measure represented by the matrix product ansatz.
	
	\smallskip
	
	The strategy of our proof is a general construction showing that any rational model can be represented as an {\it Hidden Markov Model} on a suitable enlarged state space \cite{GI1}. When the matrices are finite, the enlarged state space is finite while it is infinite when the matrices are infinite. The enlarged state space is obtained by considering the original sequence together with the labels of rows and columns of the matrices. On this enlarged state space, the measure is obtained by a product of elements of one single matrix with blocks given by the original collection of matrices. We can now transform this single matrix into a stochastic one by a conjugation operation using the Perron eigenvalues and eigenvectors. This operation is a generalized Doob transform and it is related to the Hammersley-Clifford theorem of equivalence between Gibbs and Markov measures. The enlarged probability measure is therefore naturally described in terms of a Markov bridge.  We point out that the rate functionals obtained by this procedure will always have a variational representation as an infimum over rate functionals for enlarged Markov measures.
	
	\smallskip
	In \cite{GI1} an additional structure for rational models is enlightened, namely that measures of this kind can be represented as mixtures of inhomogeneous product measures with the parameters of the product measures distributed according to a Markov bridge. This fact gives an additional structure to the rate functionals that can indeed be represented by variational problems over a double entropy functional. This different and simpler representation will be discussed in \cite{GGIM}. This relates the matrix product ansatz to the mixture representations of
	some stationary non equilibrium states \cite{CFGGT, dMFG, CFF, GRV, RV} 
	\smallskip
	
	The paper is organized as follows.
	
	In Section 2 we introduce notation, define the empirical measures, give a formal definition of the rational models, explain the matrix product ansatz with an emphasis on the boundary driven TASEP and finally illustrate the enlarging construction that is at the core of our arguments; the basic construction is explained in \cite{GI1} and we here discuss some details that were postponed there.
	
	In Section 3 we prove a general large deviations Theorem for the finite case with a variational representation of the  rate functional for both the algebraic and the spatial empirical measures.
	We also illustrate some special cases when the variational formula can be made explicit. A different representation will be discussed in \cite{GGIM}.
	
	In Section 4 we apply the ideas developed in the finite setting to the boundary driven TASEP reobtaining the results in \cite{Der5,Derr7} and in particular getting a representation of the rate functional as an infimum similar to that in \cite{Bryc}.
	
	In Appendix we collect some useful classic results and some auxiliary Lemmas.

	\section{The framework}\label{sec:frame}
	
	\subsection{Notation}
	
	We consider $A,B$ finite or countable sets and introduce notation using mainly $A$ but the same notation holds also for $B$ and other sets. We define $A^*=\cup_{N=1}^{+\infty}A^N$ the set of non empty finite words in the alphabet $A$.
	If $\eta\in A^*$ is a word with symbols written in the alphabet $A$, with $\eta\in A^N$, we write $|\eta|=N$ and use also the extended notation $\eta=(\eta_1,\dots ,\eta_N)$, so that $\eta_i\in A$ is the letter at position $i$ for the word $\eta$. For integers $n<m\leq |\eta|$ we denote by $\eta_n^m:=(\eta_n,\eta_{n+1},\dots \eta_m)$ the finite portion of the string $\eta$ between the indices $n$ and $m$. In some constructions, terms of the form $\eta_i$ with $i>|\eta|$ may appear; in this case we mean that the index of the position of the letter in the string is taken modulus $|\eta|$, i.e. $\eta_{i+N}=\eta_i$ for any $i\in \mathbb Z$ when $|\eta|=N$. Given $\eta,\xi\in A^*$, we denote by $\eta\xi\in A^*$ the concatenation of the two words. 
	\smallskip

	Consider $M$ a $B \times B$ square matrix. We say that $M$ is non-negative if $M_{b,b'}\geq 0$ for any $b,b'\in B$.
	We associate a directed graph $(B,E)$ to the matrix $M$ considering $(b,b')\in E$ if and only if $M_{b,b'}>0$.
	The matrix $M$ is called irreducible if $(B,E)$ is strongly connected. Given a state $b\in B$ we call the period of $b$ as the greatest common divisor of the set $\{k\in \mathbb N : M^k_{b,b}>0\}$. For irreducible matrices the period of each state is the same. We call {\it aperiodic} a matrix having the period of each state equal to 1.
	
	\smallskip
	
	We call respectively $\mathcal M^1(S)$ and $\mathcal M^+(S)$ the set of probability measures and positive measures on a given set $S$ endowed with the topology induced by the weak convergence, i.e. $\mu_n\to\mu$ when $\int fd\mu_n\to \int fd\mu$ for any continuous and bounded $f:S\to \mathbb R$. 
	
	Given $\nu\in \mathcal M^1(A^k)$ we say that the measure satisfies a finite stationary condition if
	\begin{equation}\label{staz-fin}
		\sum_{a\in A}\nu(\eta a)=\sum_{a\in A}\nu (a\eta )\,,\qquad \forall \eta\in A^{k-1}\,,
	\end{equation}
	and we call $\mathcal M^1_{st}(A^k)$ the set of such probability measures. 
	
	\smallskip

	Consider the real interval $[0,1]\subseteq \mathbb R$.
	We call a function $\pi:[0,1]\to \mathbb R$ absolutely continuous if there exists a function $\dot\pi\in L^1([0,1])$
	such that for almost all $x\in[0,1]$ we have $\pi(x)=\int_0^x\dot\pi(y)dy$. We call $\mathcal{AC}([0,1])$ the set of absolutely continuous functions. 
	
	\subsection{Empirical measures}
	
	\subsubsection{Algebraic empirical measures}
	The set $A$ will always be finite while $B$ may be either finite or countably infinite.
	We assume, for simplicity, that in the finite case $A:=\left\{0,1, \dots \ |A|-1\right\}$ and $B:=\left\{0,1,\dots ,|B|-1\right\}$ while instead in the infinite countable case we have $B=\mathbb N\cup \{0\}=:\mathbb N_0$.
	
	Given $\eta\in A^*$ we define its {\it algebraic empirical measure} of order one as
	\begin{equation}
		\label{em1}
		\hat\nu^{1}=\hat\nu^{1}[\eta]:=\frac{1}{|\eta|} \sum_{i=1}^{|\eta|}\delta_{\eta_i}\in \mathcal M^1(A)\,,
	\end{equation}
	where  $\delta_a$ is the delta measure at $a\in A$. When $\eta\in A^N$ that is $|\eta|=N$ we call $\hat\nu^{1}_N$ the empirical measure \eqref{em1}. The empirical measures are constructed starting from a sample, that is $\eta$ in this case, and later on it will be a pair $(\eta,\zeta)$; we denote such dependence on the sample using squared parenthesis.
	
	For any positive integer $k$ we define the {\it algebraic empirical measure of order k} as
	\begin{equation}
		\label{emk}
		\hat\nu^{k}=\hat\nu^{k}[\eta]:=\frac{1}{|\eta|} \sum_{i=1}^{|\eta|}\delta_{\eta_i^{i+k-1}}\in \mathcal M^1_{st}(A^k)\,,
	\end{equation}
	that, by definition, satisfies the finite version of the stationary condition \eqref{staz-fin}; again when $|\eta|=N$ we use the notation $\hat\nu^{k}_N$ for \eqref{emk}.
	
	For any fixed $\eta$, we have that $\left(\hat \nu^{k}\right)_{k\in \mathbb N}$ is a sequence of compatible probability measures, i.e.
	\begin{equation}
		\sum_{a\in A}\hat\nu^{k+1}(\xi a)=\hat\nu^{k}(\xi)\,, \qquad \forall \xi\in A^k\,,\ \forall k\geq 1\,,
	\end{equation}
	and, by Kolmogorov theorem, are therefore the marginals of a shift invariant measure on $\mathcal M^1(A^{\mathbb N})$ called the stationary process (see for example \cite{denH} for details and notice that you do not need to have $k<|\eta|$).

	\subsubsection{Spatial empirical measures}
	In this section we introduce some natural empirical measures having a spatial structure. Consider the interval $[0,1]$ as reference geometric space. Given $\eta\in A^*$, we define the corresponding {\it spatial empirical measure} as
	\begin{equation}\label{ems}
		\hat\pi=\hat\pi[\eta]:=\frac{1}{|\eta|}\sum_{i=1}^{|\eta|}\eta_i\delta_{i/|\eta|}\,.
	\end{equation}
	We have that $\hat\pi\in \mathcal M^+([0,1])$ is a positive measure on the interval $[0,1]$. When $|\eta|=N$
	we call $\hat\pi_N$ the spatial empirical measure \eqref{ems}. Equivalently, we can define the empirical measure by its action on continuous functions $f:[0,1]\to \mathbb R$
	\begin{equation}
		\int_{[0,1]}f\,d \hat\pi_N=\frac 1N\sum_{i=1}^N\eta_if(i/N)\,.
	\end{equation}
	
	We introduce also a {\it generalized spatial empirical measure} defined by
	\begin{equation}\label{gesm}
		\hat \Pi^k=\hat\Pi^k[\eta]:=\frac{1}{|\eta|}\sum_{i=1}^{|\eta|}\delta_{\eta_i^{i+k-1}}\delta_{i/|\eta|}(dx)\,,\qquad k=1,2,\dots
	\end{equation}
	where we added the symbol $dx$ to distinguish $\delta_y(dx)$, the spatial delta measure at $y\in[0,1]$, from $\delta_{\eta}$, the element of $\mathcal M^1(A^k)$ concentrated on $\eta$ such that $|\eta|=k$. We have that $\hat\Pi^k\in \mathcal M([0,1]; \mathcal M^+(A^k))$, i.e. it is a measure valued measure on the interval $[0,1]$, that can be equivalently characterized by its action on continuous vector fields. Consider $\mathfrak f=(\mathfrak f_\zeta)_{\zeta\in A^k}$ where for any $\zeta$ we have that $\mathfrak f_\zeta:[0,1]\to \mathbb R$ is a continuous function. We have then
	$$
	\hat\Pi^k[\eta](\mathfrak f):=\frac{1}{|\eta|}\sum_{i=1}^{|\eta|}\mathfrak f_{\eta_i^{i+k-1}}(i/|\eta|)\,.
	$$
	Given a measurable $S\subset [0,1]$ and $\eta\in A^k$ we have that 
	$$\hat\Pi^k(S)=\frac{1}{|\eta|}\sum_{\left\{i:i/|\eta|\in S\right\}}\delta_{\eta_i^{i+k-1}}\in \mathcal M^+(A^k)$$ and $\hat \Pi^k(S;\zeta)\in \mathbb R_+$ is the weight associated with $\zeta\in A^k$ by $\hat\Pi^k(S)\in \mathcal M^+(A^k)$. Some simple relations are $\hat \Pi^k([0,1])=\hat\nu^k$ and $\hat\pi(S)=\sum_{a\in A} \hat\Pi^1(S;a)a$.
	
	The topology considered for the spatial empirical measures is always that induced by the weak convergence. Consider a measure $\Pi^k\in \mathcal M([0,1]; \mathcal M^+(A^k))$; when $\Pi^k$ is absolutely continuous with respect to Lebesgue measure we use the same symbol for the measure and its density, i.e. $\Pi^k=\Pi^k(x)dx$, and we have $\Pi^k(S)=\int_S\Pi^k(x)dx$. For almost any $x\in[0,1]$, we have that the density $\Pi^k(x)\in \mathcal M^+(A^k)$.
	
	\subsection{Rational models}
	
	We have for each $a\in A$ a non-negative $|B|\times|B|$ matrix $M^{(a)}$ whose rows and columns are labeled by the elements of $B$. We recall also that the term non-negative here refers to the fact that all the elements $M^{(a)}_{b,b'}$ are non negative. We call $M:=\sum_{a\in A} M^{(a)}$. 
	
	To denote vectors we use the Bra-Ket formalism of quantum mechanics since this is the one traditionally used for the matrix product ansatz.
	We denote by Ket vectors $\ket x, \ket y$ the column vectors
	$$
	\ket x=\begin{bmatrix} x(1) \\ x(2) \\ \vdots \\ x(|B|)\,,
	\end{bmatrix}, \qquad \ket y=\begin{bmatrix} y(1) \\ y(2) \\ \vdots \\ y(|B|)\ \\ 
	\end{bmatrix}.
	$$ 
	We will consider only vectors with positive coordinates and will also use the shorthand $\ket x=\left(x(b)\right)_{b\in B}$, $\ket y=\left(y(b)\right)_{b\in B}$.  We call $\bra x,\bra y$ the corresponding row vectors obtained by transposition. We do not consider complex conjugation since we restrict to the cases of vectors with real and positive elements. In the case of infinite alphabets we will consider suitable summability conditions so that the model is well defined; we will discuss such cases afterwards. 
	
	We define a probability measure $\mu_N$ on $A^N$ by
	\begin{equation}\label{inv-rational}
		\mu_N( \eta):=\frac{\bra y\prod_{i=1}^NM^{(\eta_i)}\ket x}{Z_N}\,,\qquad \eta\in A^N\,,
	\end{equation}
	where $Z_N$ is the normalization factor
	\begin{equation}
		Z_N:=\bra y  M^N \ket x\,.
	\end{equation}
	Note that the measure and the  normalization factor depend on $\ket x,\ket y$ and $(M^{(a)})_{a\in A}$. In cases where we need to underline the dependence on such factors, we will make the dependence explicit, for example by writing $\mu_N=\mu_N^{x,y}$.
	
	\subsection{Enlarging the state space}
	Here we illustrate the construction in \cite{GI1} in order to understand the statistics and the large deviations of the measures $\mu_N$; in the next section we will apply it to some specific cases. We define a coupling  measure $\mathcal C_N$ on $A^N\times {B}^{N+1}$ (i.e. $\mathcal C_N\in \mathcal M^1(A^N\times{B}^{N+1})$) defined as follows. Consider $\eta=(\eta_1,\dots ,\eta_N)\in A^N$ and $\zeta=(\zeta_1,\dots,\zeta_{N+1})\in {B}^{N+1}$. We define 
	\begin{equation}\label{ingrassata}
		\mathcal C_N(\eta,\zeta):=\frac{y(\zeta_1)\left(\prod_{i=1}^NM^{(\eta_i)}_{\zeta_i,\zeta_{i+1}}\right)x(\zeta_{N+1})}{Z_N}\,.
	\end{equation}
	Multiplying by $|A|$ the normalization constant, the above measure can also be thought of as a probability measure on $A^{N+1}\times {B}^{N+1}$ whose value does not depend on $\eta_{N+1}$. 
	By construction we have
	\begin{equation}
		\sum_{\zeta\in {B}^{N+1} }\mathcal C_N(\eta,\zeta)=\mu_N(\eta)\,.
	\end{equation}
	It turns out that a very natural approach to compute large deviations rate functionals for the measures $(\mu_N)_{N\in \mathbb N}$ is to prove large deviations principles for the measures $(\mathcal C_N)_{N\in \mathbb N}$ and then apply the contraction principle.

	\subsection{The matrix product anstatz}\label{sec:MPA}
	
	The Matrix Product Ansatz (MPA) is a remarkable algebraic method to get a combinatorial representation of the invariant measure of a variety of non equilibrium stochastic particle systems; it has a long history starting from \cite{Der1,Der2,Der3,Schutz}, with a large number of applications to many different models (see for example \cite{reviewhold,review} for some reviews and \cite{NS} for connections to combinatorics and polymer models). The power of the method is shown, for example, by the fact that it allows to deduce large deviations principles for non equilibrium models (see for example the breakthrough papers \cite{Der4,Der5,Derr6,Derr7,Enaud}). For simplicity, we illustrate the method for a system of particles with state space $\{0,1\}^N$ and then show how our general approach applies. A probability measure represented by the matrix product ansatz often corresponds to a rational model with an infinite countable alphabet $B$. There is not a complete equivalence since in the MPA the matrices
	may have complex entries, even if they are often real and non negative, thus corresponding to a rational model.
	
	\smallskip
	
	We consider the boundary driven totally asymmetric exclusion process (TASEP) on a one dimensional chain $\{1,2\dots ,N\}$ with N sites. This is a continuous time Markov process with generator given by
	\begin{align}\label{gen-ex}
		\mathcal L_N f(\eta)&=\sum_{i=1,\dots N-1}\eta_i(1-\eta_{i+1})\left[f(\eta^{i,i+1})-f(\eta)\right]\\
		&+\alpha\left[f(\eta^{1,+})-f(\eta)\right]+\beta\left[f(\eta^{N,-})-f(\eta)\right]\,,
	\end{align}
	where
	\begin{equation}
		\eta^{i,j}_k=\left\{
		\begin{array}{ll}
			\eta_k & \textrm{if}\ k\neq i,j \\
			\eta_i & \textrm{if}\ k=j \\
			\eta_j & \textrm{if}\ k=i
		\end{array}
		\right.
	\end{equation}
	and
	\begin{equation}
		\eta^{i,+}_k=\left\{
		\begin{array}{ll}
			\eta_k & \textrm{if}\ k\neq i\\
			1 & \textrm{if}\ k=i\,,
		\end{array}
		\right.
		\qquad 
		\eta^{i,-}_k=\left\{
		\begin{array}{ll}
			\eta_k & \textrm{if}\ k\neq i\\
			0 & \textrm{if}\ k=i\,,
		\end{array}
		\right.
	\end{equation}
	and $\alpha,\beta$ with $\leq\alpha,\beta\leq 1$ are positive parameters.
	Informally, we have a chain of $N$ sites where on each site at most one particle can be located; the variable $\eta_i$ assumes the value $1$ when there is a particle at $1\leq i\leq N$ and assumes the value zero when the site $i$ is empty. Particles jump with rate one to the nearest neighbor site to the right (i.e. from site $i$ to site $i+1$) and the jump is suppressed if this site is already occupied; finally, particles are injected with rate $\alpha$ in the leftmost site $1$, when it is empty, and particles in the rightmost site $N$ are destroyed with rate $\beta$.
	
	This is an important non equilibrium model whose invariant measure can be represented by the matrix product ansatz. To be aligned with the usual notation for this method we call the two matrices $(M^{(a)})_{a=0,1}$ also as  $M^{(0)}=E$ and $M^{(1)}=D$. To have the invariant measure the following relations have to be satisfied for some vectors $\ket x$ and $\bra y$
	\begin{equation}\label{comm}
		\left\{
		\begin{array}{l}
			DE=D+E \,,\\
			D \ket x=\frac 1\beta \ket x\,, \\
			\bra y E=\frac 1\alpha \bra y\,.
		\end{array}
		\right.
	\end{equation}
	For this model the matrix product ansatz states that the invariant measure of \eqref{gen-ex} is given by the rational measure \eqref{inv-rational} when the matrices satisfy \eqref{comm}. Equations \eqref{comm} do not identify uniquely the matrices $(M^{(a)})_{a=1,2}$, and there are several different representations for different values of the parameters; all of them, apart from very special cases, are obtained by infinite countable matrices. This is somehow unavoidable; as we will show in the next sections, the LDP principles are quite different in the cases of rational models with finite or infinite matrices. In the infinite case we will discuss only an example of matrix representation of the relations \eqref{comm}. The general strategy can, however, be applied quite generally.

	\section{The finite case}
	In this section we consider the case when the matrices $(M^{(a)})_{a\in A}$ are finite.
	Since our main aim is to illustrate the general method, we consider the simplest situation where all the matrices $(M^{(a)})_{a\in A}$ are irreducible and aperiodic and all the components of the vectors $\ket x,\ket y$ are strictly positive. We will discuss more general cases in a forthcoming paper \cite{GGIM} where also the mixture structure in \cite{GI1} will be used to get a different representation. We refer to  \cite{G2} for results in the general case, obtained by still another different approach. Our approach can also be adapted to the case of infinite matrices, as we will see in the next section. For simplicity, we compute the large deviations for some lower order algebraic and spatial empirical measures. Again with a more detailed analysis the result can be extended to higher order.
	
	\smallskip
	
	\subsection{Enlarged Markov bridges}\label{sec:bridge}
	
	Given the collection of matrices $(M^{(a)})_{a\in A}$ we define an $\left(|A||B|\right)\times \left(|A||B|\right)$ matrix $\mathfrak M$ having rows and columns labeled by elements of $A\times B$ and defined as
	\begin{equation}\label{defT}
		\mathfrak M_{(a,b),(a',b')}:=M^{(a)}_{b,b'}\,, \qquad a,a'\in A\,;\ b,b'\in B\,.
	\end{equation}
	We discuss some preliminary lemmas.
	
	\begin{lemma}\label{lemmaT}
		If the $B\times B$ matrices $(M^{(a)})_{a\in A}$ are irreducible and aperiodic, then the $(A\times B)\times (A\times B)$ matrix $\mathfrak M$ defined in \eqref{defT} is irreducible and aperiodic.
	\end{lemma}
	\begin{proof}
		Since all the matrices $(M^{(a)})_{a\in A}$ are irreducible then also $M=\sum_{a\in A}M^{(a)}$ is irreducible. Given any pair $b_I,b_F\in B$ then there exists a path $(b^*_1,\dots b^*_\ell)$ such that $b^*_1=b_I$, $b^*_\ell=b_F$ and $M_{b^*_i,b^*_{i+1}}>0$. Let us call $a^*_i\in A$, $i=1,\dots,\ell-1$ an element such that $M^{(a^*_i)}_{b^*_i,b^*_{i+1}}>0$. We need to show that for any $(a,b),(a'.b')\in A\times B$ there exists a path $(a_i,b_i)_{i=1}^k$ that satisfies  $(a_1,b_1)=(a,b)$, $(a_k,b_k)=(a',b')$ and $\mathfrak M_{(a_i,b_i),(a_{i+1},b_{i+1})}>0$. Since by irreducibility there exists $b_I$ such that $M^{(a)}_{b,b_I}$ is strictly positive, we define $(a_1,b_1)=(a,b)$, $(a_2,b_2)=(a^*_1,b_I)$ where $a^*_i$ and $b^*_i$ are the sequences associated, as before, with a path from $b_I$ to $b_F=b'$ and then $(a_i,b_i)=(a^*_{i-1},b^*_{i-1})$, $i=2, \dots k$ (observe that $a^*_\ell$ can be fixed equal to $a'$), and we have $k=\ell+1$. Aperiodicity follows easily by aperiodicity of all the $(M^{(a)})_{a\in A}$.
	\end{proof}
	
	Recall the Perron-Frobenius theorem in Section \ref{PF-finite} in the appendix. We have the following relations among maximal eigenvalues and eigenvectors of the matrices $M$ and $\mathfrak M$.
	
	\begin{lemma}\label{lemmae}
		If all the $(M^{(a)})_{a\in A}$ are irreducible and aperiodic, we have that the maximal eigenvalue $\Lambda$ of the matrix $\mathfrak M$ is equal to the maximal eigenvalue $\lambda$ of the matrix $M$, moreover if $\ket{e}$ is the positive maximal eigenvector for $M$ then the positive maximal eigenvector $\ket{\varepsilon}$ of $\mathfrak M$ is given by
		\begin{equation}\label{maxeL}
			\varepsilon(a,b):=\frac 1\lambda \sum_{b'\in B}M^{(a)}_{b,b'} e(b')\,.
		\end{equation} 
	\end{lemma}
	\begin{proof}
		By the previous lemma, $\mathfrak M$ is irreducible and aperiodic and therefore it has a unique maximal eigenvalue with a positive eigenvector by Perron Frobenius theorem (see theorem \ref{thPFf}). By irreducibility and the positivity of $\ket{e}$, we have that the vector $\ket{\varepsilon}$, as defined by formula \eqref{maxeL}, has all positive elements. 
		
		Taking the sum over $a\in A$ on both sides of \eqref{maxeL} we observe that
		\begin{equation}\label{serve}
			\sum_{a\in A}\varepsilon(a,b)=\frac 1\lambda \sum_{b'\in B}M_{b,b'}\underline e(b')= e (b)\,.
		\end{equation}
		To finish the proof it is enough to show that $\ket\varepsilon$ is an  eigenvector of $\mathfrak M$ with eigenvalue $\lambda>0$ and this follows from
		\begin{align}
			&\sum_{a',b'}\mathfrak M_{(a,b),(a',b')}\varepsilon(a',b')=\sum_{a',b'}M^{(a)}_{b,b'}\varepsilon(a',b')\\
			& =\sum_{b'}M^{(a)}_{b,b'} e(b')=\lambda \varepsilon(a,b)\,,
		\end{align}
		where we used \eqref{serve} and the definition \eqref{maxeL}.
	\end{proof}
	
	We introduce the matrices:
	\begin{equation}\label {matriceS}
		\left\{
		\begin{array}{l}
			S_{b,b'}= \frac{1}{\lambda} e(b)^{-1} M_{b,b'}e(b')\,, \\
			\mathfrak S_{(a,b),(a',b')}=\frac 1\lambda \varepsilon^{-1}(a,b)\mathfrak M_{(a,b),(a',b')}\varepsilon(a',b')
		\end{array}
		\right.
	\end{equation}
	that by Lemmas \ref{lemmaT}, \ref{lemmae} and \ref{stoclemma}, are stochastic, irreducible and aperiodic.
	Introducing the $B\times B$ diagonal matrix $E$ having diagonal elements $ E_{b,b}=e(b)$ and the  $(A\times B)\times (A\times B)$ diagonal matrix $\mathcal E$ having diagonal elements $\mathcal E_{(a,b),(a,b)}=\varepsilon(a,b)$, the formula \eqref{matriceS} can be compactly written as 
	\begin{equation}\label{matriceScompact}
		\left\{
		\begin{array}{l}
			S=\frac 1\lambda E^{-1}M E\,, \\
			\mathfrak S=\frac 1\lambda \mathcal E^{-1}T\mathcal E.
		\end{array}
		\right.
	\end{equation}

	We will not use the following lemma on the unique invariant measures directly, but we state and prove it anyway.

	\begin{lemma}
		If $\theta=(\theta_b)_{b\in B}$ is the invariant measure of the stochastic matrix $S$, then
		\begin{equation} \label{invmeas}
			\Theta(a,b)= \theta(b) \frac{\varepsilon(a,b)}{e(b)}\,, \qquad a\in A, b\in B\,,
		\end{equation}
		is the invariant measure of $\mathfrak S$. 
	\end{lemma}
	\begin{proof}
		Let us introduce the matrices $\left(S^{(a)}\right)_{a\in A}$ defined by $S_{b,b'}^{(a)}:= \frac{1}{\lambda} e(b)^{-1} M^{(a)}_{b,b'}e(b')$.
		A direct relation between $\mathfrak S$ and the matrices $(S^{(a)})_{a\in A}$ is
		\begin{align*}
			\mathfrak S_{(a,b),(a',b')} &= \frac{1}{\lambda}\varepsilon^{-1}(a,b)M^{(a)}_{b,b'}\varepsilon(a',b') \\ & = \left(\frac{\varepsilon(a,b)}{e(b)} \right)^{-1} S^{(a)}_{b,b'} \left(\frac{\varepsilon(a',b')}{e(b')} \right) \,.
		\end{align*} 
		Using the latter relation we can show that \eqref{invmeas} is the invariant measure of $\mathfrak S$: 
		\begin{align*}
			&\sum_{a,b} \Theta(a,b) \mathfrak S_{(a,b),(a',b')} = \sum_{a,b} \Theta(a,b) \left(\frac{\varepsilon(a,b)}{e(b)}\right)^{-1} S^{(a)}_{b,b'} \left( \frac{\varepsilon(a',b')}{e(b')} \right) \\&= \sum_{a,b} \theta(b)S^{(a)}_{b,b'} \left( \frac{\varepsilon(a',b')}{e(b')} \right) =  \sum_{b \in B} \theta(b) S_{b,b'}\left( \frac{\varepsilon(a',b')}{e(b')} \right) \\&=\theta(b') \left( \frac{\varepsilon(a',b')}{e(b')} \right) =\Theta(a',b')\,.
		\end{align*}
	\end{proof}

	We can now illustrate the general structure of the probability measures $\mathcal C_N$ on the enlarged state space;
	to simplify the notation we consider a generic state space $\mathbb A$ instead that $A\times B$.
	Consider a Markov chain on $\mathbb A$ with transition probability
	$P$. Let $g,f: \mathbb A\to \mathbb R^+$ be two positive functions, and define the measure on $\mathbb A^{N+1}$ by
	\begin{equation}\label{bridge}
		\mathbb P_N^{f,g}(\xi)=\frac{f(\xi_1)\prod_{i=1}^NP_{\xi_i,\xi_{i+1}}g(\xi_{N+1})}{Z_N}\,.
	\end{equation}
	A special case is when $f(\cdot)=\delta_{\cdot, a}$ and $g(\cdot)=\delta_{\cdot, a'}$ with $a,a'\in \mathbb A$ and $\delta_{\cdot,\cdot}$ denoting the Kronecker delta; in this case, the measure gives a positive weight only to trajectories that satisfy $\xi_1=a$ and $\xi_{N+1}=a'$. In this special case, the measure is called the Markov bridge with transition probability $P$ between $a$ and $a'$ in the discrete time window $1,\dots ,N+1$ and the normalization constant is given by
	$Z_N=P^N_{a,a'}$. The process corresponds to the Markov chain with transition matrix $P$ conditioned to start in $a$ at time $1$ and to end in $a'$ at time $N+1$. This is a Markov measure but with time non homogeneous transition probabilities. The general case \eqref{bridge} corresponds to a measure for which the initial and final points $(\xi_1,\xi_{N+1})$ are sampled according to a suitable distribution, and then the initial and final points are connected by a Markov bridge starting at $\xi_1$ and finishing at $\xi_{N+1}$.
	The initial and final points $(\xi_1,\xi_{N+1})$ are sampled according to the distribution 
	\begin{equation}\label{mbr}
		m(\xi_1,\xi_{N+1})=\frac{f(\xi_1)P^N_{\xi_1,\xi_{N+1}}g(\xi_{N+1})}{Z_N}\,.
	\end{equation}
	
	The Markov nature of a Markov bridge can be verified observing that 
	$$\mathbb P_N^{f,g}\left(\xi_{k+1}=a_{k+1}|\xi_{k}=a_k, \dots, \xi_1=a_1\right)$$ depends just on $a_{k+1}$ and $a_k$ and the direct computation reveals the time non-homogeneous nature of the transition probabilities
	\begin{equation}
		\mathbb P_N^{f,g}\left(\xi_{k+1}=a_{k+1}|\xi_{k}=a_k, \dots, \xi_1=a_1\right)=\frac{P_{a_k,a_{k+1}}\sum_{\xi_{N+1}}P^{N-k}_{a_{k+1},\xi_{N+1}}g(\xi_{N+1})}{\sum_{\xi_{N+1}}P^{N-k+1}_{a_{k},\xi_{N+1}}g(\xi_{N+1})}\,.
	\end{equation}
	If we need to specify the transition matrix $P$ of the Markov bridge measure we use the notation $\mathbb P_N^{P,f,g}$.
	
	By using the transformation \eqref{matriceS}, we have that the enlarged measure \eqref{ingrassata}, is indeed the Markov bridge measure $\mathbb P_N^{\mathfrak S,f,g}$ associated with the $(A\times B)\times (A\times B)$ stochastic matrix $\mathfrak S$ defined in \eqref{matriceS} and with the functions $f,g;A\times B\to \mathbb R$ defined by
	\begin{equation*}
		\left\{
		\begin{array}{l}
			f(a,b)=y(b)\varepsilon (a,b)\,\\
			g(a,b)=\varepsilon^{-1}(a,b)x(b)\,.
		\end{array}
		\right.
	\end{equation*}
	
	\subsection{The general statement}
	
	We can now state our main result in the finite case. We denote by $\mathcal C^2$ a generic element 
	of $\mathcal M^2_{stat}\left((A\times B)^2\right)$ and by $\mathcal C^1$ its one marginal, i.e.
	\begin{equation}\label{1-margC}
		\mathcal C^1(a,b):=\sum_{(a',b')\in A\times B}\mathcal C^{2}\left[(a,b),(a',b')\right]=
		\sum_{(a',b')\in A\times B}\mathcal C^{2}\left[(a',b'),(a,b)\right]\,,
	\end{equation}
	where the second equality follows from the finite stationarity of $\mathcal C^{2}$.
	We instead denote by $\nu^2$ a generic element 
	of $\mathcal M^2_{stat}(A^2)$. We refer to \cite{denH} for the basic definitions on large deviations.
	\begin{theorem}
		Assume that $|A|<+\infty$, every $(M^{(a)})_{a\in A}$ is irreducible and aperiodic and $\ket x, \ket y$ have strictly positive components. When $\eta$ is distributed according to $\mu_N$ in \eqref{inv-rational}, we have that $\hat\nu^{2}_N$, defined in \eqref{emk}, satisfies a large deviation principle on $\mathcal M^1_{st}(A^2)$ with speed $N$ and  rate functional
		\begin{equation}\label{formulafinita}
			I^2\left(\nu^{2}\right)=\inf \left\{\sum_{a,a'\in A; b,b'\in B}\mathcal C^{2}\left[(a,b),(a',b')\right]\log \frac{\mathcal C^{2}\left[(a,b),(a',b')\right]}{\mathcal C^{1}\left[(a,b)\right]M^{(a)}_{b,b'}}+\log\lambda\right\}\,,
		\end{equation}
		where the infimum is over 
		\begin{equation}\label{vincolo}
			\left\{\mathcal C^2\in \mathcal M^2_{stat}\left((A\times B)^2\right)\,:\,\sum_{b,b'\in B}\mathcal C^{2}\left[(a,b),(a',b')\right]=\nu^{2}(a,a')\,, \forall a,a'\in A\right\}
		\end{equation}
		and $\lambda$ is the maximal eigenvalue of $M$. We have also that $\hat{\Pi}^2_N$ satisfies a large deviations principle on $\mathcal M([0,1]; \mathcal M^+(A^2))$ with rate functional
		\begin{equation}\label{finitaspaz}
			J^2\left(\Pi^2\right)=\left\{
			\begin{array}{ll}
				\int_0^1 I^2\left(\Pi^2(x)\right)dx\,, & \textrm{if}\ \Pi^2=\Pi^2(x)dx\,,\\
				+\infty & \textrm{otherwise}\,,
			\end{array}
			\right.
		\end{equation}
		where $\Pi^2(x)\in \mathcal M^1_{st}(A^2)$ for almost all $x$.
	\end{theorem}
	
	\begin{proof}
		By Lemmas \ref{lemmaT}, \ref{lemmae} we have that the matrix 
		\begin{equation}\label{Sdef}
			\mathfrak S_{(a,b),(a',b')}=\frac 1\lambda \varepsilon(a,b)^{-1}\mathfrak M_{(a,b),(a',b')}\varepsilon(a',b')
		\end{equation}
		is stochastic.
		By the positivity of the vectors $\ket x,\ket y,\ket{\varepsilon}$ we have that there exist positive constants $k,K$ such that 
		$$k \mathbb P_N^{\mathfrak S, \mathrm{st}}\leq \mathbb P_N^{\mathfrak S, f,g}\leq K\mathbb P_N^{\mathfrak S, \mathrm{st}}$$ where we called $\mathbb P_N^{\mathfrak S,\mathrm{st}}$ the stationary Markov measure with transition probability $\mathfrak S$, i.e.
		\begin{equation}
			\mathbb P_N^{\mathfrak S, \mathrm{st}}(\eta,\zeta)=\Theta(\eta_1,\zeta_1)\prod_{i=1}^{N}\mathfrak S_{(\eta_i,\zeta_i),(\eta_{i+1},\zeta_{i+1})}\,,
		\end{equation}
		where $\Theta$ is the unique stationary measure of the transition matrix $\mathfrak S$ and the value of $\eta_{N+1}$ is irrelevant given the form of $\mathfrak S$.
		We can therefore apply Lemma \ref{lemmac} deducing the large deviations for the empirical measure for $\mathbb P_N^{\mathfrak S, f,g}$ from the large deviations of the Markov measure $\mathbb P_N^{\mathfrak S,\mathrm{st}}$, that are well established (see for example \cite{denH}) since it is a finite state irreducible  stationary Markov measure. 
		
		Let us call 
		\begin{equation}\label{emprod}
			\hat {\mathcal C}^{2}_{N+1}:=\frac{1}{N+1}\sum_{i=1}^{N+1}\delta_{\big((\eta_i,\zeta_i),(\eta_{i+1},\zeta_{i+1})\big)}\in \mathcal M^2_{st}\left((A\times B)^2\right)\,,
		\end{equation}
		where as usual we have that the sums are taken modulo $N+1$. By the general theory of large deviations for Markov chains, see for example \cite{denH}, we have that \eqref{emprod} satisfies a LDP with rate functional $\mathcal I^2$ on $\mathcal M^2_{st}\left((A\times B)^2\right)$ that is defined as follows. Let $\mathcal C^{2}\in \mathcal M^2_{st}\left((A\times B)^2\right)$ and recall that $C^{1}\in \mathcal M^1(A\times B)$ is its one marginal \eqref{1-margC}; we have then \cite{denH}
		\begin{align}
			\mathcal I^2(\mathcal C^{2})&=\sum_{a,a',b,b'}\mathcal C^{2}\left[(a,b),(a',b')\right]\log \frac{\mathcal C^{2}\left[(a,b),(a',b')\right]}{\mathcal C^{1}\left[(a,b)\right]\mathfrak S_{(a,b),(a',b')}}\,,\\
			&=\sum_{a,a',b,b'}\mathcal C^{2}\left[(a,b),(a',b')\right]\log \frac{\mathcal C^{2}\left[(a,b),(a',b')\right]}{\mathcal C^{1}\left[(a,b)\right]M^{(a)}_{b,b'}}+\log\lambda\,,
		\end{align}
		where the second inequality follows since $\mathcal C^{2}\in \mathcal M^2_{st}\left((A\times B)^2\right)$.
		We can now apply the contraction principle obtaining for  $\nu^{2}\in\mathcal M^2_{stat}(A^2)$
		\begin{equation}
			I^2(\nu^{2})=\inf_{\{\sum_{b,b'}\mathcal C^{2}\left[(a,b),(a',b')\right]=\nu^{2}(a,a')\,, \forall a,a'\}} \mathcal I^2(\mathcal C^{2})\,,
		\end{equation}
		and this finishes the proof for the algebraic case. The proof of \eqref{finitaspaz} can also be obtained following classic arguments for Markov measures.
	\end{proof}
	Large deviations principles for lower order empirical measures can be obtained by contraction, LDP for higher order empirical measures can be obtained following the same stategies used above and using the classic arguments for Markov measures (see, for example, \cite{denH}) on the enlarged state space and then applying the contraction principle.
	
	In the general case it is difficult to get a more explicit form of the rate functional with respect to \eqref{formulafinita}. We discuss, however, some special cases.
	
	\subsection{Examples}
	
	\subsubsection{A general form}
	The critical condition in the minimization \eqref{formulafinita} is obtained considering several constraints that the measure $C^{2}$ has to satisfy. The first collection of constraints is given by \eqref{vincolo}, while the other collection of constraints is obtained imposing $C^{2}\in \mathcal M_{st}\left((A\times B)^2\right)$. To each constraint there corresponds a Lagrange multiplier. By a direct computation, the critical conditions are equivalent to
	\begin{equation}\label{constraints}
		\frac{\mathcal C^{2}\left[(a,b),(a',b')\right]}{\mathcal C^{1}\left[(a,b)\right]}=k^{-1}M^{(a)}_{b,b'}p(a,a')\frac{\gamma_{a',b'}}{\gamma_{a,b}}\,,
	\end{equation}
	where the positive matrix $p$ encodes the set of Lagrange multipliers for the constraints \eqref{vincolo}, while the numbers $\gamma$ and the constant $k$ are the Lagrange multipliers for the second group of constraints. Since on the left-hand side we have a stochastic matrix, we obtain that the minimizer is obtained when $\left(\gamma_{a,b}\right)_{a\in A,b\in B}$ is the maximal eigenvector of the positive $(A\times B)\times (A\times B)$ matrix $\left(M^{(a)}_{b,b'}p(a,a')\right)_{a,a'\in A}^{b,b'\in B}$ and $k$ is the corresponding maximal eigenvalue.
	
	\smallskip
	Inserting \eqref{constraints} as the minimizer in \eqref{formulafinita} we obtain, using that $\mathcal C^{2}$ is stationary and has $\nu^{2}$ as $A\times A$ marginal (that are the constraints \eqref{vincolo}),
	\begin{equation}\label{finitap}
		I(\nu^{2})=\sum_{a,a'}\nu^{2}(a,a')\log p(a,a')+\log (\lambda/k)\,.
	\end{equation}
	It can be shown that in \eqref{constraints} and \eqref{finitap} we can restrict to matrices $p(a,a')$ that are stochastic. Similar considerations can also be applied for the spatial empirical measure.

	\subsubsection{Parallel eigenvectors}
	We consider the special case when 
	\begin{equation}\label{special-parallel}
		M^{(a)}_{b,b'}=m(a)S^{(a)}_{b,b'}\frac{\phi(b)}{\phi(b')}
	\end{equation}
	where $(S^{(a)})_{a\in A}$ is a collection of $B\times B$ stochastic matrices, $(m(a))_{a\in A}$ is a collection of positive numbers and $(\phi(b))_{b\in B}$ is a positive vector. This corresponds to a family of positive matrices $M^{(a)}$, each having the same maximal eigenvector $\ket \phi$ with corresponding maximal eigenvalue $m(a)$. Indeed, Lemma \ref{stoclemma} implies that all the positive matrices having $\ket \phi$ as a maximal eigenvector are of the form \eqref{special-parallel} for some stochastic matrix $S^{(a)}$. We have the following.
	\begin{lemma}\label{poilotogliamo}
		For matrices $(M^{(a)})_{a\in A}$ of the form \eqref{special-parallel} we have that the maximal eigenvalue and eigenvector $k$, $\ket \gamma$ of the $(A\times B)\times (A\times B)$ matrix $\left(M^{(a)}_{b,b'}p(a,a')\right)_{a,a'\in A}^{b,b'\in B}$ are related to $\mu$, $\ket \psi$ those of the $A\times A$ matrix $\left(p(a,a')m(a')\right)_{a,a'\in A}$ by the relations 
		\begin{equation}
			\left\{
			\begin{array}{l}
				k=\mu\,, \\
				\gamma(a,b)=\psi(a)m(a)\phi(b)\,.
			\end{array}
			\right.
		\end{equation}
		Moreover, we have that the maximal eigenvalue $\lambda$ of the $B\times B$ matrix $M$ is $\sum_am(a)$ with corresponding eigenvector $\ket \phi$.
	\end{lemma}
	
	\begin{proof}
		Since a positive eigenvector is necessarily the one associated with the maximal eigenvalue, by a direct computation we have
		\begin{align}
			& \sum_{a',b'}M^{(a)}_{b,b'}p(a,a')\gamma(a',b') \\
			&=\sum_{a',b'}m(a)S^{(a)}_{b,b'}\frac{\phi(b)}{\phi(b')}p(a,a')m(a')\psi(a')\phi(b')\\
			&=\phi(b)m(a)\sum_{a'}p(a,a')m(a')\psi(a')=\mu \gamma(a,b)\,.
		\end{align}
		The last statement follows also by the direct computation
		\begin{align*}
			\sum_{b'}M_{b,b'}\phi(b')=\phi(b)\sum_{a}m(a)\sum_{b'}S^{(a)}_{b,b'}=\left(\sum_a m(a)\right)\phi(b)\,.
		\end{align*}
	\end{proof}
	We can now deduce the following.
	\begin{proposition}
		When the matrices $(M^{(a)})_{a\in A}$ have all the same Perron eigenvector, then the large deviations 
		rate functional $I^2$ in \eqref{formulafinita} coincides with the large deviations rate functional
		for the pair empirical measure of a sequence of i.i.d. random variables taking values on $A$ and having probability distribution $\mathbb P(a)=\frac{m(a)}{\sum_{a'\in A}m(a')}$, $a\in A$; i.e. we have
		\begin{equation}\label{2finpar}
			I(\nu^{2})=\sum_{a,a'}\nu^{2}(a,a')\log\frac{\nu^{2}(a,a')}{\nu^{1}(a)\mathbb P(a')}\,.
		\end{equation}
	\end{proposition}
	\begin{proof}
		Inserting the eigenvalues obtained in the previous lemma into the relation \eqref{constraints} for the minimizer we get
		$$
		\mathcal C^{2}((a,b),(a',b'))=\mathcal C^{1}(a,b)k^{-1}S^{(a)}_{b,b'}p(a,a')\frac{\psi(a')m(a')}{\psi(a)}\,,
		$$ 
		and we can determine $p$ satisfying \eqref{vincolo} just summing over $(b,b')\in B^2$ the above equality getting
		\begin{equation}\label{pqs}
			p(a,a')=\frac{k \nu^{(2)}(a,a')\psi(a)}{\nu^{(1)}(a)\psi(a')m(a')}\,.
		\end{equation}
		Inserting this value in \eqref{finitap}, recalling that $\nu^{2}\in\mathcal M_{st}\left(A^2\right)$ and the value of $\lambda$ obtained in Lemma \ref{poilotogliamo}, by a direct computation and using Lemma \ref{poilotogliamo} we obtain \eqref{2finpar},
		which is the large deviations rate functional for the pair empirical measure for a product measure on $A$, assigning probability $\mathbb P(a)$ to each element $a\in A$ (see for example \cite{denH}).
	\end{proof}
	As in the general case, also in this case the rate functional for the empirical measures of order different from 2 and for the spatial empirical measures can be obtained by the same approach.

	\subsubsection{Stochastic matrices}
	
	A special case of the situation discussed in the previous subsection is when all the matrices $(M^{(a)})_{a\in A}$ are stochastic. In this case all the Perron eigenvalues are equal to one and the distribution $\mathbb P(a)=\frac{1}{|A|}$ is the uniform one. In this case, the results can also be obtained directly as follows.
	The first observation is that, when all the matrices $(M^{(a)})_{a\in A}$ are stochastic, the measure $\mu_N^{ 1, 1}$ associated with the vectors $\ket x,\ket y$ having all the coordinates identically equal to one, is the uniform measure on $A^N$. This follows by a direct check.
	
	Let us call $|\ket x|_M=\sup_a x(a)$ and $|\ket x|_m=\inf_a x(a)$ that are both finite and strictly positive, we use a similar notation for $y$ too. We have
	$$
	|x|_m|y|_m\mu_N^{\underline 1,\underline 1}\leq \mu_N^{x,y}\leq |x|_M|y|_M\mu_N^{\underline 1,\underline 1}\,,
	$$
	and from Lemma \ref{lemmac} we deduce
	
	\begin{equation}\label{unik}
		I^k(\nu^{k})=\left\{
		\begin{array}{ll}
			\sum_{a_1^k\in A^k}\nu^{k}(a_1^k)\left(\log\nu^{k}(a_1^k)-k\log|A| \right) & \textrm{if} \ \nu^{k}\in \mathcal M^1_{st}(A^k) \\
			+\infty & \textrm{otherwise}
		\end{array}
		\right.
	\end{equation}
	\smallskip
	since the rate functional \eqref{unik} is exactly the rate functionals associated with a product of uniform measures.
	Similar considerations can also be done for the spatial empirical measures.

	\section{The infinite case}
	For Markov measures in the finite case an irreducibility condition is sufficient to obtain a LDP both for the algebraic and spatial empirical measures. This is not the case for a countable infinite alphabet, where suitable conditions have to be satisfied, see \cite{dv1-4}. In the example we consider only LDP for spatial empirical measures can be directly obtained, and we therefore restrict to that case. We apply the general results concerning eigenvalues and eigenvectors discussed in the finite case; a proof of the general statements in the infinite case is more subtle, but for our specific example we proceed by direct computation.
	
	\smallskip
	
	We consider now an example of a rational model with an infinite alphabet $B$; this is the boundary driven TASEP discussed in Section \ref{sec:MPA}. In particular, we consider a special representation of the matrices satisfying the conditions \eqref{comm}, which is valid when the parameters vary in the region $\frac{(1-\alpha)(1-\beta)}{\alpha\beta}<1$, see for example \cite{K}. This corresponds to the following matrices.

	Consider the sets $A=\{0,1\}$, $B=\mathbb{N}\cup \{0\}=:\mathbb N_0$ and the two $B\times B$ matrices:
	
	\begin{equation}\label{matrici}
		M^{(0)}=\begin{bmatrix}
			1 & 0 & 0 & \cdots & 0 & \cdots \\ 1 & 1 & 0 & \cdots & 0 & \cdots \\ 0&1&1&0&\cdots &\cdots \\ 0& 0 &1&1&0 &\cdots \\ \vdots&\vdots&0&\ddots&\ddots&\ddots \\ \vdots&\vdots&\vdots&\ddots&\ddots&\ddots
		\end{bmatrix}\text{,} \ \	M^{(1)}=\begin{bmatrix}
			1 & 1 & 0 & \cdots & 0 & \cdots \\ 0 & 1 & 1 & 0& \cdots & \cdots \\ 0&0&1&1& 0& \cdots & \\ 0& 0 &0&1&1 &\ddots \\ \vdots&\vdots&\vdots&0&\ddots&\ddots \\ \vdots&\vdots&\vdots&\ddots&\ddots&\ddots
		\end{bmatrix}.
	\end{equation}
	These matrices have the following left and right eigenvectors:
	\begin{equation*}
		\left\{
		\begin{array}{l}
			\bra{y}M^{(0)}=\frac{\bra{y}}{\alpha} \\
			M^{(1)}\ket{x}=\frac{\ket{x}}{\beta}
		\end{array}
		\right.
	\end{equation*}
	with:
	\begin{equation}\label{y}
		\bra{ y}= \hat{k} \left [1, \frac{1-\alpha}{\alpha}, \left(\frac{1-\alpha}{\alpha}\right )^2,\left(\frac{1-\alpha}{\alpha}\right )^3,... \right]
	\end{equation}
	\begin{equation}\label{x}
		\ket x=\hat{k}\begin{bmatrix} 1 \\ \frac{1-\beta}{\beta} \\ \left(\frac{1-\beta}{\beta}\right )^2 \\ \left(\frac{1-\beta}{\beta}\right )^3 \\ \vdots
		\end{bmatrix}
	\end{equation}
	where $\hat{k}=\sqrt{(\alpha+\beta-1)/\alpha \beta}$ is a normalization factor that we fix in such a way that $\braket{ y| x}=1$. 
	
	The matrix $M=M^{(0)}+M^{(1)}$ is the infinite matrix with all the elements on the main diagonal equal to $2$, the ones on the lower and on the upper diagonal equal to 1 and all the other entries equal to 0, i.e.
	\begin{equation}\label{M}
		M_{b,b'}=\left\{
		\begin{array}{ll}
			2\delta_{b,b'}+\delta_{b,b'+1}+\delta_{b,b'-1}\,, & b\geq 1\,,\\
			2\delta_{b,b'}+\delta_{b,b'+1}\,, & b=0\,.
		\end{array}
		\right.
	\end{equation}
	The matrix $M$ coincides with the matrix discussed in Section \ref{eigeninf} with $\alpha=2$, $\beta_1=\beta_2=1$, the Perron eigenvalue is therefore $\lambda=4$ and the corresponding positive eigenvector is:
	\begin{equation}\label{eigenn}
		\ket{e}= \begin{bmatrix} 1 \\ 2 \\ \vdots \\ n \\ \vdots
		\end{bmatrix}
	\end{equation}
	\newline
	
	We define the matrix $\mathfrak M_{(a,b),(a',b')}=M^{(a)}_{b,b'}$, with $(a,b) \in \{0,1\} \times \mathbb{N}_0$, which has the same Perron eigenvalue $\Lambda=\lambda=4$ as the matrix $M$ and the corresponding eigenvector is defined according to \eqref{maxeL} by the formula:
	\begin{equation*}
		\varepsilon(a,b)= \frac{1}{4} \sum_{b' \in \mathbb N_0} M^{(a)}_{b,b'}(b'+1)\,,
	\end{equation*}
	recalling that $\ket e=\Big(b+1\Big)_{b\in \mathbb N_0}$ is the eigenvector \eqref{eigenn} of the matrix $M$ associated to the eigenvalue $\lambda=4$. By a direct computation we obtain:
	\begin{equation*}
		\varepsilon(a,b)= \frac{1}{4} (2b+2a+1)\,, \qquad a\in\{0,1\}\,, \ b\in \mathbb N_0\,.
	\end{equation*}
	We can now write the stochastic $\Big(\{0,1\} \times \mathbb{N}_0 \Big)\times \Big(\{0,1\} \times \mathbb{N}_0\Big) $ matrix $\mathfrak S$ explicitly:
	\begin{equation}\label{MS}
		\mathfrak S_{(a,b),(a',b')}=\frac{1}{4} \frac{2(b'+a')+1}{2(b+a)+1} M^{(a)}_{b,b'}\,, \qquad a\in\{0,1\}\,,\, b\in \mathbb N_0\,.
	\end{equation}
	Recalling the form of the matrix $M$ in \eqref{M}, we have that all the elements of the matrix $\mathfrak S$ are equal to zero except for those of the form specified below. The non zero elements associated with transitions from the state $(0,0)$ are
	\begin{equation*}
		\mathfrak S_{(0,0),(0,0)}= \frac{1}{4}\,, \qquad  \mathfrak S_{(0,0),(1,0)}=\frac{3}{4}\,;
	\end{equation*}
	the non zero elements associated with transitions from the state $(1,0)$ are
	\begin{equation*}
		\mathfrak S_{(1,0),(1,0)}=S_{(1,0)(0,1)}=\frac{1}{4}\,, \qquad \mathfrak S_{(1,0),(0,0)}=\frac{1}{12}\,, \qquad \mathfrak S_{(1,0),(1,1)}=\frac{5}{12}\,;
	\end{equation*}
	the non zero elements associated with transitions from the state $(0,b)$ with $b\geq 1$ are
	\begin{equation*}
		\mathfrak S_{(0,b),(0,b)}=\mathfrak S_{(0,b),(1,b-1)}= \frac{1}{4}\,, \ \mathfrak S_{(0,b)(1,b)}=\frac{2b+3}{4(2b+1)}\,, \ \mathfrak S_{(0,b),(0,b-1)}=\frac{2b-1}{4(2b+1)}\,;
	\end{equation*}
	the non zero elements associated with transitions from the state $(1,b)$ with $b\geq 1$ are
	\begin{equation*}
		\mathfrak S_{(1,b),(1,b)}=\mathfrak S_{(1,b),(0,b+1)}=\frac{1}{4}\,, \ \mathfrak S_{(1,b),(0,b)}=\frac{2b+1}{4(2b+3)}\,, \ \mathfrak S_{(1,b),(1,b+1)}=\frac{2b+5}{4(2b+3)}\,.
	\end{equation*}
	The transition graph associated with the matrix $\mathfrak S$ is drawn in Figure \ref{vera}. The corresponding Markov chain is a special random walk on two infinite lines with a left boundary at the origin of the two lines. Transitions are possible between the two lines. The transition probabilities of the random walk are spatially non homogeneous. According to the discussion in Section \ref{sec:bridge}, the enlarged measure of the boundary driven TASEP is a Markov bridge of length $N+1$ of this special random walk.
	\begin{figure}
		\centering
		\includegraphics[width=10cm]{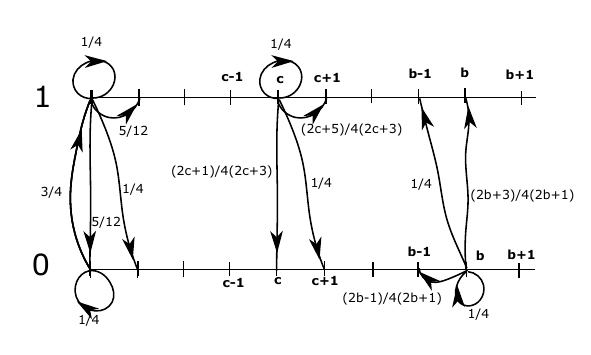}
		\caption{A graphical representation of the transition probabilities of the Markov chain $\mathfrak S$ on $\{0,1\}\times \mathbb N_0$. The coordinates ${\bf c,b}$ on the axis $\mathbb N_0$ are written in boldface to distinguish them from the values of the transition probabilities.}\label{vera}
	\end{figure}
	
	\smallskip
	
	It will be useful in order to simplify the proofs to introduce an effective random walk on $\{0,1\} \times \mathbb{N}_0$ with a simpler transition matrix $\mathcal S$. This is again a random walk on the two infinite lines but it is spatially homogeneous with a different behavior just at the boundary sites. The transition matrix is given by
	\begin{equation}\label{toy}
		\left\{
		\begin{array}{ll}
			\mathcal S_{(0,b)(a', b+s)}=\frac 14\,,& b\neq 0;\ a'=0,1;\ s=0,-1;\\
			\mathcal S_{(1,b)(a', b+s)}=\frac 14\,, &a'=0,1;\ s=0,+1; \\
			\mathcal S_{(0,0)(a',0)}=\frac 12\,, & a'=0,1\,;
		\end{array}
		\right.
	\end{equation}
	we give a graphical representation of the transition probabilities in Figure \ref{effettivo}.
	\begin{figure}
		\centering
		\includegraphics[width=10cm]{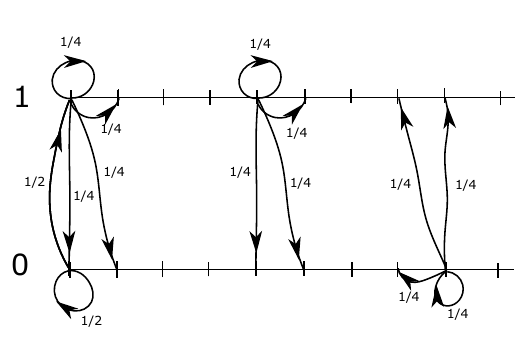}
		\caption{A graphical representation of the transition probabilities of the effective Markov chain $\mathcal S$ on $\{0,1\}\times \mathbb N_0$.}\label{effettivo}
	\end{figure}
	The stochastic matrix $\mathcal S$ is almost related by a generalized Doob transform to the stochastic matrix $\mathfrak S$, the relation is violated only for transitions exiting from the state $(0,0)$. By a direct computation, we obtain the relation
	\begin{equation}\label{FnonF}
		\mathcal S_{(a,b)(a',b')}=\mathfrak S_{(a,b)(a',b')}\frac{2(a+b)+1}{2(a'+b')+1}2^{\id\big((a,b)=(0,0)\big)}
		=\frac 14 M^{(a)}_{b,b'}2^{\id\big((a,b)=(0,0)\big)}\,.
	\end{equation}
	The possibility to work with this simple homogeneous random walk is due to the special form of the matrices $M^{(0)},M^{(1)}$.
	
	\smallskip

	Let us recall  that $\mathbb P^{\mathfrak S, f,g}_{N}$ and $\mathbb P^{\mathcal S,f',g'}_{N}$ denote Markov bridges with transition matrix $\mathfrak S$ and $\mathcal S$, respectively. We call $\mathfrak Z_N(f,g)$ and $\mathcal Z_N(f',g')$ the corresponding normalization factors appearing in the definition \eqref{bridge}. We denote by $(\eta,\zeta)=(\eta_k,\zeta_k)_{k=1}^{N+1}$ a trajectory on $\{0,1\}\times \mathbb N_0$ of length $N+1$ and having positive probability with respect to the Markov measures. We call $\mathbb P^{\mathfrak S}_{N}$ and $\mathbb P^{\mathcal S}_{N}$ the corresponding Markov measures with fixed initial condition; we have for example
	$\mathbb P^{\mathfrak S}_{N}(\eta,\zeta):=\prod_{i=1}^{N}\mathfrak S_{(\eta_i,\zeta_i),(\eta_{i+1},\zeta_{i+1})}$, and likewise for the other matrix.
	Finally, considering \eqref{em1} for the product space $A\times B$, we call $\hat{\nu}^1_N[\eta,\zeta]=\frac{1}{N+1} \sum_{i=1}^{N+1}\delta_{(\eta_i,\zeta_i)}\in \mathcal M^1(A\times B)$. The introduction of the simplified homogeneous Markov measure with transition probability $\mathcal S$ is due to the fact that we can deduce LDP's for Markov measures with transition probability $\mathfrak S$ from the corresponding ones for the matrix $\mathcal S$. This follows from the next lemma. Let us introduce the function $h:\{0,1\}\times \mathbb N_0\to \mathbb R$ defined by $h(a,b):=2(a+b)+1$.
	\begin{lemma}\label{lemmastorto}
		When $\mathfrak Z_N, \mathcal Z_N<+\infty$ we have
		\begin{equation}\label{piupura}
			\mathbb P^{\mathfrak S, f,g}_{N}(\eta,\zeta)=\mathbb P^{\mathcal S,f/h,gh}_{N}(\eta,\zeta)2^{-N\hat{\nu}_N^1(0,0)}\frac{\mathcal Z_N(f/h,gh)}{\mathfrak Z_N(f,g)}\,.
		\end{equation}
	\end{lemma}
	\begin{proof}
		By a direct computation we have for any fixed $(\eta,\zeta)$
		$$
		\mathbb P^{\mathfrak S}_{N}(\eta,\zeta)=\mathbb P^{\mathcal S}_{N}(\eta,\zeta)\frac{h(\eta_{N+1},\zeta_{N+1})}{h(\eta_1,\zeta_1)}2^{-N\hat{\nu}_N^1(0,0)}\,.
		$$
		Multiplying both sides by $f(\eta_1,\zeta_1)$ and $g(\eta_{N+1},\zeta_{N+1})$ and inserting the constants we get 
		\eqref{piupura}.
	\end{proof}

	\subsection{Large deviations asymptotics}
	In this section we illustrate the strategy to prove a spatial LDP for the Markov measure \eqref{ingrassata} for the matrices \eqref{matrici} and the vectors \eqref{x}, \eqref{y}. We prove first the corresponding result for the simple homogeneous Markov measure with transition $\mathcal S$, and then use Lemma \ref{lemmastorto}
	to deduce the result for the original matrix $\mathfrak S$. 
	
	We will only outline some steps that are technically involved for random walks with boudary, but follow a general line \cite{EK, SW}.
	We will then apply the contraction principle to get the LD rate functional for the boundary driven TASEP obtaining 
	the well-known functional.
	
	\subsubsection{The probabilistic structure}
	
	The simple homogeneous random walk with transition matrix $\mathcal S$ has some special features. Since $\sum_{b'\in \mathbb N_0}\mathcal S_{(a,b),(a',b')}=\frac 12$ for any $a,a'\in\{0,1\}$ and $b\in \mathbb N_0$, we deduce that, under the measure $\mathbb P^{\mathcal S}_N$, the variables $\eta_2^{N+1}$ are i.i.d. all having distribution $\mathcal B_{1/2}$. Moreover, we can construct samples of the Markov measure $\mathbb P_N^{\mathcal S}$ using sequences of i.i.d. random vectors. More precisely, we introduce $(X^I_i,Y^I_i)_{i\in \mathbb N}$ a sequence of i.i.d. random vectors such that
	$$
	\mathbb P\left[(X^I_i,Y^I_i)=(a,b)\right]=\left\{
	\begin{array}{ll}
		\frac 14 & \textrm{if}\ (a,b)=(0,0),(0,-1),(1,0),(1,1)\,,\\
		0 & \textrm{otherwise}\,.
	\end{array}
	\right.
	$$
	We also introduce $(X^B_i,Y^B_i)_{i\in \mathbb N}$ a sequence of i.i.d. random vectors such that
	$$
	\mathbb P\left[(X^B_i,Y^B_i)=(a,b)\right]=\left\{
	\begin{array}{ll}
		\frac 14 & \textrm{if}\ (a,b)=(1,0),(1,1)\,,\\
		\frac 12 & \textrm{if}\ (a,b)=(0,0)\,,\\
		0 & \textrm{otherwise}\,.
	\end{array}
	\right.
	$$
	The two collection of i.i.d. random vectors are also independent of each other.
	The random variables $(X^I_i,Y^I_i)_{i\in \mathbb N}$ have to be used to construct the path of the Markov measure $\mathbb P_N^{\mathcal S}$ in the bulk, while the random variables $(X^B_i,Y^B_i)_{i\in \mathbb N}$ have to be used in the boundary.
	
	We call $\mu^I,\mu^B\in \mathcal M^1\big(\{0,1\}\times\{-1,0,1\}\big)$ the distribution of the pair $(X^I_i,Y^I_i)$ and $(X_i^B,Y_i^B)$, respectively; namely,
	\[
	\left\{
	\begin{array}{l}
		\mu^I_{0,0}=\mu^I_{0,-1}=\mu^I_{1,0}=\mu^I_{1,1}=\frac 14 \,,\\
		\mu^I_{0,1}=\mu^I_{1,-1}=0\,,
	\end{array}
	\right.
	\qquad
	\left\{
	\begin{array}{l}
		\mu^B_{1,0}=\mu^B_{1,1}=\frac 14 \,,\\
		\mu^B_{0,0}=\frac 12 \,,\\
		\mu^B_{1,-1}=\mu^B_{0,1}=0  \,.
	\end{array}
	\right.
	\]
	We iteratively define the sequence of random variables $(\eta,\zeta)$ distributed according to $\mathbb P_N^{\mathcal S}$. The variables $(\eta_1,\zeta_1)$ are fixed and we generate $\zeta_2$ as follows. If $(\eta_1,\zeta_1)=(0,0)$ then $\zeta_2=0$; otherwise $\zeta_2=\zeta_1$ or $\zeta_2=\zeta_1+2\eta_1-1$ with equal probability and independently of the sequences of i.i.d. vectors. Once that $\zeta_2$ is determined we generate the whole sequence as follows.
	For any $i>1$ we have that, when $\zeta_i>0$ then $(\eta_i,\zeta_{i+1}-\zeta_i)=(X^I_i,Y^I_i)$, when instead $\zeta_i=0$ then 
	$(\eta_i,\zeta_{i+1}-\zeta_i)=(X^B_i,Y^B_i)$. By a direct inspection it is possible to see that the sequence
	$(\eta_i,\zeta_i)_{i=1}^{N+1}$, constructed in this way, has law $\mathbb P_N^{\mathcal S}$.
	In the case of an initial condition with the variables $(\eta_1,\zeta_1)$ distributed as $\mathcal B_{1/2}(\eta_1)\delta_{\zeta_1}$ the whole collection of random variables can be constructed by the same iteration. In this case, as input to start the iteration, we only need to assign the value $\zeta_1$. In the Figure \ref{tab} the diagram shows how to construct iteratively the variables starting form $\zeta_1$ and the arrows exiting from each $\zeta_i$ show which are the variables that can be constructed directly using the i.i.d. sequences. As follows from the diagram in the Figure \ref{tab}, the sequence $\zeta$ alone is Markovian; this property is related to the representation in terms of mixtures of inhomogeneous product measures in \cite{GI1}.
	\begin{figure}
		\begin{tikzpicture}[node distance=1.4cm]
			\node (start1) [startstop] {$\zeta_1$};
			\node (in1) [io, above of=start1] {$\eta_1$};
			\node (start2) [startstop, right of=start1] {$\zeta_2$};
			\draw [arrow] (start1) -- (start2);
			\draw [arrow] (start1) -- (in1);
			\node (start3) [startstop, right of=start2] {$\zeta_3$};
			\node (in2) [io, above of=start2] {$\eta_2$};
			\node (start4) [startstop, right of=start3] {$\zeta_4$};
			\node (in3) [io, above of=start3] {$\eta_3$};
			\node (in4) [io, above of=start4] {$\eta_4$};
			\draw [arrow] (start2) -- (start3);
			\draw [arrow] (start2) -- (in2);
			\draw [arrow] (start3) -- (start4);
			\draw [arrow] (start3) -- (in3);
			
			\node (start5) [startstop,right of=start4] {$\zeta_5$};
			\node (in5) [io, above of=start5] {$\eta_5$};
			\node (start6) [startstop, right of=start5] {$\zeta_6$};
			\draw [arrow] (start5) -- (start6);
			\draw [arrow] (start5) -- (in5);
			\node (start7) [startstop, right of=start6] {$\zeta_7$};
			\node (in6) [io, above of=start6] {$\eta_6$};
			\node (start8) [startstop, right of=start7] {$\zeta_8$};
			\node (in7) [io, above of=start7] {$\eta_7$};
			\node (in8) [io, above of=start8] {$\eta_8$};
			\draw [arrow] (start6) -- (start7);
			\draw [arrow] (start6) -- (in6);
			\draw [arrow] (start7) -- (start8);
			\draw [arrow] (start7) -- (in7);
			\draw [arrow] (start4) -- (start5);
			\draw [arrow] (start4) -- (in4);
			\draw [arrow] (start8) -- (in8);
			\draw [arrow] (start8) -- (11.2,0);
		\end{tikzpicture}
		\caption{The dependency structure of the random variables. The arrows exiting from $\zeta_k$ indicate that the variables $\eta_k$ and $\zeta_{k+1}$ can be constructed using the i.i.d variables.}
		\label{tab}
	\end{figure}
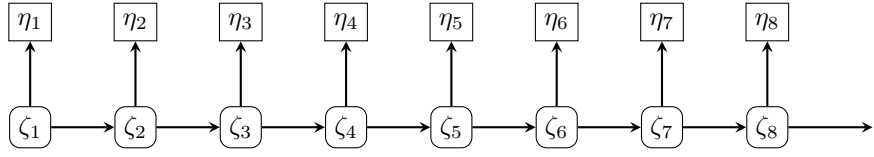
	
	\subsubsection{The empirical observables}
	
	It is convenient to study the large deviations asymptotic in terms of some empirical observables that we are going to introduce. First of all we consider two empirical measures defined by
	\begin{equation}\label{sanmartinoB}
		\begin{cases}
			\hat{\Pi}^B_N[\eta,\zeta]:=\frac{1}{N}\sum_{i=1}^N\id(\zeta_i=0)\delta_{(\eta_i,\tilde\zeta_i)}\delta_{i/N}\\
			\hat{\Pi}^I_N[\eta,\zeta]:=\frac{1}{N}\sum_{i=1}^N\id(\zeta_i>0)\delta_{(\eta_i,\tilde\zeta_i)}\delta_{i/N}\,,
		\end{cases}
	\end{equation}
	where we use the random variables $\tilde\zeta_i:=\zeta_{i+1}-\zeta_i$. The two empirical measures in \eqref{sanmartinoB} are random elements of
	$\mathcal M\Big([0,1]; \mathcal M^+\big(A\times\{-1,0,+1\}\big)\Big)$. Some other natural empirical measures that can be constructed starting from \eqref{sanmartinoB} are
	\begin{equation}\label{sanmartino}
		\hat{\Pi}_N[\eta,\zeta]:=\frac{1}{N}\sum_{i=1}^N\delta_{(\eta_i,\tilde\zeta_i)}\delta_{i/N}\in \mathcal M\Big([0,1]; \mathcal M^+\big(A\times\{-1,0,+1\}\big)\Big)\,,
	\end{equation}
	and 
	\begin{equation}\label{piB}
		\hat{\pi}^B_N[\zeta]:=\frac{1}{N}\sum_{i=1}^N\id(\zeta_i=0)\delta_{i/N}\in \mathcal M^+([0,1])\,.
	\end{equation} 
	The empirical measure \eqref{sanmartino} can be obtained by
	$\hat{\Pi}_N=\hat{\Pi}_N^B+\hat{\Pi}_N^I$ while the empirical measure \eqref{piB} is obtained by
	$\hat{\pi}^B_N=\sum_{a}\sum_{b}\hat{\Pi}^B_N\left(a, b\right)$, where, to better explain again the notation, $\hat{\Pi}^B_N \in \mathcal M\Big([0,1]; \mathcal M^+\big(A\times\{-1,0,+1\}\big)\Big)$ is given by $\left(\hat{\Pi}^B_N\left(a,b\right)\in \mathcal M\Big([0,1]; \mathbb R\Big)\right)_{a\in A}^{ b=-1,0,+1}$. We use similar notation for other measures.
	
	The above empirical measures depend on the variable $\eta,\zeta$ and we underline such a dependence inside squared parenthesis $[\cdot ]$. 
	
	The empirical measure \eqref{ems} associated with the configuration of the TASEP model is obtained starting from \eqref{sanmartino} by
	\begin{equation}\label{imp}
		\hat \pi_N[\eta]=\sum_{a}\sum_{ b}\hat \Pi_N[\eta,\zeta]\left(a, b\right)a\,;
	\end{equation}
	our goal is to deduce a LDP for \eqref{imp}
	by contraction from a LD on a larger state space. 
	
	It will be useful to consider the following relation for the empirical measure appearing in Lemma \ref{lemmastorto}
	\begin{equation}\label{bunny}
		\hat \nu_N(0,0)=\int_{[0,1]} \sum_a \sum_{ b}(1-a)\hat \Pi_N^B(a,b; dx)\,.
	\end{equation}
	
	We introduce a last empirical object and then proceed to describe their continuous counterpart, scaling limits and large deviations. This is a piecewise linear path $\hat z_N=\left(\hat z_N(x)\right)_{x\in[0,1]}$ that is defined setting $\hat z_N\left(\frac iN\right)=\frac{\zeta_{i+1}}{N}$, $i=0, \dots , N$ and then defining it at the remaining points by linear interpolation. We have the following relation
	\begin{equation}\label{calcint}
		\hat z_N\left(\frac jN\right)-\hat z_N\left(\frac iN\right)=\int_{(i/N,j/N]}\sum_{a=0,1}\sum_{ b=-1,0,1}\hat{\Pi}_N(a,b; dx)b\,.
	\end{equation}
	Note that if $i=0, j=N$ we have on the left hand side above $\hat z_N(1)-\hat z_N(0)$ and on the right hand side
	the integral with respect to the whole interval giving a continuous correspondence.
	
	\smallskip
	
	We are interested in studying the LD for the joint empirical variables 
	\begin{equation}\label{empvar}
		\left(\hat z_N,\hat{\Pi}_N^B,\hat{\Pi}_N^I\right)\in C([0,1])\times\mathcal M\Big([0,1]; \mathcal M^+\big(A\times\{-1,0,+1\}\big)\Big)^2 \,,
	\end{equation}
	when $(\eta,\xi)$ are distributed according to the measure \eqref{ingrassata} with matrices and vectors defined in \eqref{matrici}, \eqref{y}, \eqref{x}, and where we consider the space of continuous functions with the topology induced by the supremum norm and the product space with the product topology.
	
    Let $\mathbb P^{\mathcal S,*}_N$ denote the Markov measure with initial distribution $\mathcal B_{1/2}(\eta_1)\delta_{\zeta_1}$. Since the probability of each sequence $(\eta,\zeta)$ depends only on the empirical measures in \eqref{sanmartinoB}, the measure $\mathbb P^{\mathcal S,*}_N$ can be compactly written as
	\begin{equation}\label{bellama}
		\mathbb P^{\mathcal S,*}_N=\prod_{a,b}\left[\mu^B(a,b)^{\left(N\int_{[0,1]}\,,\hat{\Pi}^B_N(a,b)\right)}
		\mu^I(a,b)^{\left(N\int_{[0,1]}d\,\hat{\Pi}^I_N(a,b)\right)}\right]\,,
	\end{equation}
	where we recall that in the product $a\in\{0,1\}$ and $b\in\{-1,0,1\}$.

	According to the definition \eqref{bridge}, as discussed at the end of Section \ref{sec:bridge}, measure \eqref{ingrassata} corresponds to the Markov bridge $\mathbb P^{\mathfrak S,f,g}$ associated with the transition matrix $\mathfrak S$ and the two functions $f,g$ determined by the vectors $\ket y$ and $\ket x$ and defined by 
	\begin{equation}\label{fgT}
		\left\{
		\begin{array}{l}
			f(a,b):=\frac{(1-\alpha)^{b}(2a+2b+1)}{4\alpha^{b}}\,,\\
			g(a,b):=\frac{4(1-\beta)^{b}}{\beta^{b}(2a+2b+1)}\,.
		\end{array}
		\right.
	\end{equation} 
	We will then deduce the LD for the empirical measure \eqref{ems} when $\eta$ is distributed according to the invariant measure of the TASEP by contractions.

	We will use Lemma \ref{lemmastorto} and deduce LD for $\mathbb P^{\mathfrak S,f,g}$ from LD for $\mathbb P^{\mathcal S,f/h,gh}_{N}$. By the expression
	\eqref{fgT} we have that the functions $f/h,gh$ do not depend on the variable $a$ so that we have
	\begin{equation}\label{bridgecompact}
		\mathbb P^{\mathcal S,f/h,gh}_{N}(\eta,\zeta)=\left(\frac{1-\alpha}{\alpha}\right)^{\zeta_1}\mathbb P^{\mathcal S,*}_N(\eta,\zeta)\left(\frac{1-\beta}{\beta}\right)^{\zeta_{N+1}} \left(\mathcal Z_N\right)^{-1}\,,
	\end{equation}
	where $\mathcal Z_N$ is the same normalization constant appearing in \eqref{piupura}.

	\subsubsection{The rate functionals}

	We now describe some joint LD rate functionals associated with the empirical observables we described before when the variables $(\eta,\zeta)$ are distributed according to different measures. The LD rate functionals are defined on the space where the empirical triple of observables lives, namely the one introduced in \eqref{empvar}. The LD rate functionals are concentrated on absolutely continuous paths and measures; this follows from the fact that for any configuration $(\eta,\zeta)$, any $(a,b)\in \{0,1\}\times \{-1,0,1\}$ and any continuous function $G$ we have
	$$
	\left|\int_{[0,1]} \Pi_N(a,b;dx)G(x)\right|\leq \frac{1}{N}\sum_{i=1}^N\left|G(i/N)\right|\,,
	$$
	so that any limit measure $\Pi$ satisfies $\left|\int_{[0,1]} \Pi(a,b;dx)G(x)\right|\leq \int_{[0,1]}\left|G(x)\right|dx$ and moreover
	$$
	\mathbb P^{\mathcal S,*}_N\left(\left|\int_{[0,1]} \Pi_N(a,b;dx)G(x)\right|\geq \int_{[0,1]}\left|G(x)\right|+\epsilon\right)=0\,,
	$$
	for $N$ large enough.
	This simple estimate, and similar ones for any empirical measure, allows to show that the LD rate functional is identically $+\infty$ on measures and paths that are not absolutely continuous with respect to the Lebesgue measure or whose density is not bounded by one. We therefore only need to show the form of the rate functional on the closed subset of absolutely continuous measures with densities bounded by one. Moreover, by relation \eqref{calcint} we have
	\begin{equation}\label{pace}
		\mathbb P^{\mathcal S,*}_N\left(\sup_{x,y}\left|\hat z_N\left(x\right)-\hat z_N\left(y\right)-\int_{[x,y]}\sum_{a=0,1}\sum_{ b=-1,0,1}\hat{\Pi}_N(a,b; dz)b\right|>\epsilon\right)=0\,, 
	\end{equation}
	for any $\epsilon$ and $N$ big enough.
	We deduce, therefore, that the rate functional has a domain that is concentrated on the closed set of absolutely continuous measures and paths where the following relations are satisfied.
	
	\smallskip
	
	Consider an absolutely continuous triple 
	$$\left(z, \Pi^B,\Pi^I\right)\in C([0,1])\times\mathcal M\Big([0,1]; \mathcal M^+\big(A\times\{-1,0,+1\}\big)\Big)^2\,,$$
	we will use the same notation for an absolutely continuous measure and its density.
	We write $\Pi^B= m(x)\nu^B (x)dx$, and $\Pi^I= (1-m(x))\nu^I (x)dx$  where $\nu^B(x), \nu^I(x)$ are space dependent probability measures on $\{0,1\}\times \{-1,0,+1\}$ and $0\leq m(x)\leq 1$. 
	
	This means that $\nu^B(x)=\left(\nu^B_{a, b}\right)_{a=0,1}^{ b=-1,0,+1}$ and $\sum_{a,b}\nu^B_{a,b}(x)=1$ and $\nu^I(x)=\left(\nu^I_{a, b}\right)_{a=0,1}^{ b=-1,0,+1}$ and $\sum_{a, b}\nu^I_{a, b}(x)=1$ for a.e. $x\in[0,1]$. 
	As in \eqref{sanmartino} we have 
	$$\Pi(x)dx=\Pi^B(x)dx+\Pi^I(x)dx=\big[m(x)\nu^B (x)+(1-m(x))\nu^I (x)\big]dx\,.$$ 
	It is possible to write in this form $\Pi^B$ and $\Pi^I$ since the density of $\sum_{a,b}\Pi(a,b)$ is identically one. For any $x$, the support of $\nu^B(x)$ is contained in that of $\mu^B$ and the support of $\nu^I(x)$ is contained in that of $\mu^I$.

	The continuous counterpart of \eqref{piB} is given by
	$$
	m(x)dx=\sum_{a, b}\Pi^B_{a, b}(x) dx\,.
	$$
	By \eqref{pace}, we deduce that the rate functional is finite only when the following relation is satisfied a.e.
	\begin{align}
		\dot z(x)&=\sum_a\sum_{ b}\Pi_{a, b}(x)bdx=\Pi_{1,1}(x)-\Pi_{0,-1}(x) \nonumber \\
		&= m(x)\nu^B_{1,1}(x)+(1-m(x))\left(\nu^I_{1,1}(x)-\nu^I_{0,-1}(x)\right) \,.\label{diff}
	\end{align}
	Finally, we call $\rho(x)dx$ the continuous counterpart of the empirical measure \eqref{imp} and we have
	$\rho(x)=\Pi_{1,1}(x)+\Pi_{1,0}(x)$.

	\smallskip 
	
	We summarize the basic constraints such that, when they are violated, the LD rate functionals we are going to determine are identically $+\infty$.
	
	\begin{itemize}
		\item [A)] The triple $\left(z, \Pi^B,\Pi^I\right)$ is absolutely continuous and $\Pi^B= m(x)\nu^B (x)dx$, and $\Pi^I= (1-m(x))\nu^I (x)dx$ for $x$ dependent probability measures $\nu^B, \nu^I$ and some $0\leq m(x)\leq 1$.
		\item [B)] The relation \eqref{diff} is satisfied a.e.
		\item [C)] The constraint $z(x)\geq 0$ is satisfied. 
		\item [D)] It holds $m(x)=0$ when $z(x)>0$. 
	\end{itemize} 
	Given $\nu$, $\mu$ two probability measures on $\{0,1\}\times \{-1,0,+1\}$ we call 
	$$h(\nu|\mu)=\left\{
	\begin{array}{ll}
		\sum_{a,b}\nu_{a, b}\log\frac{\nu_{a,b}}{\mu_{a,b}} & \nu \abscont \mu\,, \\
		+\infty & \textrm{otherwise}\,,
	\end{array}
	\right.
	$$ 
	their relative entropy. 
	
	\begin{proposition}\label{euno}
		Under  the initial condition $\mathcal B_{1/2}(\eta_1)\delta_{Nz_0}(\zeta_1)$, when $(\eta,\zeta)$ are distributed according to $\mathbb P^{\mathcal S,*}_N$ in \eqref{bellama},  the triple $\left(\hat z_N,\hat{\Pi}_N^B,\hat{\Pi}_N^I\right)$   satifies a LD principle with rate functional $\mathcal I^*$, that is $+\infty$ if at least one among the conditions A, B, C, D and $z(0)\neq z_0$ is violated, and it is otherwise equal to
		\begin{equation}\label{Istar}
			\mathcal I^*\Big(z,\Pi^B,\Pi^I\Big)=\int_0^1\Big[m(x)h\left(\nu^B(x)|\mu^B\right)+(1-m(x))h\left(\nu^I(x)|\mu^I\right)\Big]dx\,.
		\end{equation}
	\end{proposition}
	We remark that the dependence on the path $z$ in the rate functional appears through the constraint on the initial condition and the conditions B, C
	and D. Starting from Proposition \ref{euno} we can deduce the following LD rate functional whose proof is obtained by just applying the Varadhan lemma.
	
	\begin{lemma}\label{edue}
		When the variables $(\eta,\zeta)$ are distributed according to the Markov bridge \eqref{bridgecompact}, then the triple $\left(\hat z_N,\hat{\Pi}_N^B,\hat{\Pi}_N^I\right)$  satisfies a LD principle with a rate functional that is $+\infty$ if at least one among the constraints A, B, C and D is violated, and it is otherwise given by
		\begin{equation}
			\mathcal I^{\mathcal S}\Big(z,\Pi^B,\Pi^I\Big)=z(0)\log\frac{\alpha}{1-\alpha}+z(1)\log\frac{\beta}{1-\beta}+ \mathcal I^*\Big(z,\Pi^B,\Pi^I\Big)+c\,,
		\end{equation}
		where $c$ is a suitable constant and $\mathcal I^*$ is the rate functional in \eqref{Istar}.
	\end{lemma}

	By Lemma \ref{lemmastorto} we can now deduce the LD for the  Markov bridge $\mathbb P^{\mathfrak S, f,g}_{N}$, again using Varadhan lemma with respect to the continuous map \eqref{bunny}. By the structure of support of the measure $\mu^B$ the map on the right hand side of \eqref{bunny} for the absolutely continuous triple is given by $\int_0^1m(x)\nu^B_{0,0}(x)dx$. We therefore have the following.
	
	\begin{lemma}\label{etre}
		When the variables $(\eta,\zeta)$ are distributed according to the Markov bridge $\mathbb P^{\mathfrak S, f,g}_{N}$, then the triple $\left(\hat z_N,\hat{\Pi}_N^B,\hat{\Pi}_N^I\right)$ satisfies a LD principle with rate functional that is $+\infty$ if at least one of the constraints A, B, C and D are violated and it is otherwise given by
		
		\begin{equation}
			\mathcal I^{\mathfrak S}\Big(z,\Pi^B,\Pi^I\Big)=\mathcal I^{\mathcal S}\Big(z,\Pi^B,\Pi^I\Big)+ \left[\int_0^1m(x)\nu^B_{0,0}(x)dx\right]\log 2+C\,,
		\end{equation}
		where $C$ is a suitable constant.
	\end{lemma}
	
	\subsubsection{Strategy of the proof}
	We give an hint of the strategy of the proof of Proposition \ref{euno}; the following Lemmas \ref{edue} and \ref{etre} follow by an application of the Varadhan lemma.
	The basic property that we use is the fact the measure \eqref{bellama} depends just on the empirical measures $\hat \Pi^I_N$ and $\hat \Pi^B_N$.
	
	Let us introduce a family of perturbations of the probability measure \eqref{bellama}: consider two space dependent families of probability measures $\gamma^B(x)=(\gamma^B_{a,b}(x))_{a=0,1}^{b=-1,0,1}$ and $\gamma^I(x)=(\gamma^I_{a,b}(x))_{a=0,1}^{b=-1,0,1}$, and consider the corresponding absolutely continuous measures
	$\gamma^B(x)dx$ and $\gamma^I(x)dx$ that are elements in $\mathcal M\left([0,1], \mathcal M^1\big(A\times\{-1,0,+1\}\big)\right)$.  For any $x$, $\gamma^B(x)$ is strictly positive on the support of $\mu^B$ and $\gamma^I(x)$ is strictly positive on the support of $\mu^I$ and, moreover, the $x$ dependent components are Lipschitz continuous. For any pair $\gamma^B,\gamma^I$ we introduce a corresponding probability measure on the sequences $(\eta,\zeta)$ defined by
	
	\begin{equation}\label{genperturb}
		\mathbb P^{\gamma^I,\gamma^B}_N:= \exp\left\{N\sum_{a,b}\int_{[0,1]}\left[\hat \Pi^I_N(a,b; dx)\log \gamma^I_{a,b}(x)+\hat \Pi^B_N(a,b; dx)\log \gamma^B_{a,b}(x)\right]\right\}\,,
	\end{equation}
	and notice that \eqref{bellama} corresponds to the cases of $\gamma^B$ and $\gamma^I$ constant and constantly equal to $\mu^B$ and $\mu^I$ respectively.
	
	The probability measure \eqref{genperturb} is absolutely continuous with respect to \eqref{bellama} and we can compute the Radon-Nikodym derivative that is given by
	\begin{equation}\label{RN}
		\frac{1}{N}\log \frac{d \mathbb P^{\gamma^I,\gamma^B}_N}{d \mathbb P^{\mathcal S,*}_N}=\sum_{a,b}\int_{[0,1]}\left[\hat \Pi^I_N(a,b; dx)\log\frac{ \gamma^I_{a,b}(x)}{\mu^I_{a,b}}+\hat \Pi^B_N(a,b; dx)\log\frac{ \gamma^B_{a,b}(x)}{\mu^B_{a,b}}\right]\,.
	\end{equation}
	The proof of the large deviations principle is based on the law of large numbers for the perturbed measures according to the following fact.
	
	\begin{proposition}\label{derby}
		When $(\eta,\zeta)$ are distributed according to $\mathbb P^{\gamma^I,\gamma^B}_N$ with initial condition on the $\zeta$ variables given by $\zeta_1=Nz_0$, then the triple $\left(\hat z_N,\hat{\Pi}_N^B,\hat{\Pi}_N^I\right)$ converges in probability to the absolutely continuous triple $(z,\Pi^B,\Pi^I)$ 
		such that 
		\begin{equation}\label{com}
			\left\{
			\begin{array}{l}
				\Pi^B(x)=m(x)\gamma^B(x)\,, \\
				\Pi^I(x)=(1-m(x))\gamma^I(x)\,,
			\end{array}
			\right.
		\end{equation}
		and $z(x)$ is the solution of
		\begin{equation}\label{intreno}
			\left\{
			\begin{array}{l}
				z(0)=z_0\\
				\dot z(x)=m(x)\gamma^B_{1,1}(x)+(1-m(x))(\gamma_{1,1}^I(x)-\gamma^I_{0,-1}(x))\,.
			\end{array}
			\right.
		\end{equation}
		In the above formulas $m(x)=0$ when $z(x)>0$ or when $z(x)=0$ and $\gamma^I_{1,1}(x)-\gamma^I_{0,-1}(x)>0$, otherwise when $z(x)=0$ and $\gamma^I_{1,1}(x)-\gamma^I_{0,-1}(x)\leq 0$, then $0\leq m(x) \leq 1$ is the unique value such that
		\begin{equation}\label{unicam}
			m(x)\gamma^B_{1,1}(x)+(1-m(x))(\gamma_{1,1}^I(x)-\gamma^I_{0,-1}(x))=0\,;
		\end{equation}
		this means 
		\begin{equation}
			m(x)=\frac{\left[\gamma^I_{0,-1}(x)-\gamma^I_{1,1}(x)\right]_+}{\gamma^B_{011}(x)+\gamma^I_{0,-1}(x)-\gamma^I_{1,1}(x)}\id(z(x)=0)\,,
		\end{equation}
		where $[\cdot ]_+$ denotes the positive part.
	\end{proposition}
	\begin{proof}
		We give an outline of the proof; see \cite{EK, SW} for results of this type.
		Due to the Lipschitz continuous assumption on $\gamma^B, \gamma^I$ we have existence and uniqueness of the solution of \eqref{intreno}. We already showed that any limit triple has to be absolutely continuous and, moreover, by independent sampling any limit measure has to be of the form \eqref{com} for a suitable $m(x)$. The proof of the lemma for the intervals on which $z(x)>0$ follows directly since it is a special case of the classic fluid limit \cite{K, SW} with $m=0$. On the intervals on which $z(x)=0$ and $\gamma^I_{1,1}(x)-\gamma^I_{0,-1}(x)\leq 0$, the result follows since any scaling limit has to be absolutely continuous, by \eqref{pace} it has to satisfy \eqref{diff}, and there is an unique $m(x)$ satisfying \eqref{unicam}.
	\end{proof}
	
	\paragraph{\bf Exponential tightness}
	
	To show exponential tightness we need to construct a family of compact sets $\left(\mathcal K_\alpha\right)_{\alpha \in \mathbb R^+}$ such that
	$$
	\limsup_{N\to+\infty} \frac 1N \log P^{\mathcal S,*}_N\left((\hat z_N, \hat \Pi^B, \hat \Pi^I)\in \mathcal K_\alpha\right)< -\alpha\,.
	$$
	This statement allows to extend the following upper bound from compact set to closed ones. This is a technical issue that we do not discuss in detail and that can be treated considering $\mathcal K_\alpha=\mathcal K_\alpha^z\times\mathcal K_\alpha^B\times \mathcal K_\alpha^I$ and following the arguments in \cite{DZ} Lemma 5.1.7 and \cite{KL} Section 4 of Chapter 10.
	
	\paragraph{\bf Lower bound}
	Consider a triple $(z, \Pi^B, \Pi^I)$ that satisfies the conditions A, B, C and D and such that $\Pi^B(x)=m(x)\gamma^B(x)$ and $\Pi^I(x)=(1-m(x))\gamma^I(x)$ with $\gamma^B,\gamma^I$ that are strictly positive and Lipschitz continuous. Then, by Proposition \ref{derby}, it follows that, with respect to the measures $\mathbb P^{\gamma^I,\gamma^B}_N$ the triple $\left(\hat z_N,\hat{\Pi}_N^B,\hat{\Pi}_N^I\right)$ converges in probability to the triple $(z,\Pi^B,\Pi^I)$ and, moreover, by the formula \eqref{RN} we have
	\begin{equation}
		\lim_{N\to +\infty}\frac 1Nh\left( \mathbb P^{\gamma^I,\gamma^B}_N \Big | \mathbb P^{\mathcal S,*}_N\right) 
		=\lim_{N\to +\infty}\frac 1N \mathbb E_{\mathbb P^{\gamma^I,\gamma^B}_N}\left[\log \frac{d \mathbb P^{\gamma^I,\gamma^B}_N}{d \mathbb P^{\mathcal S,*}_N}\right] = \mathcal I^*\Big(z,\Pi^B,\Pi^I\Big)\,.
	\end{equation}
	Let $\tilde{\mathcal I}^*$ be the functional equal to $\mathcal I^*$ for measures $\Pi^B=m\nu^B$ and $\Pi^I=(1-m)\nu^I$ with $\nu^I, \nu^B$ that are strictly positive and Lipschitz continuous and  equal to $+\infty$ otherwise.
	By Theorem 3.4 in \cite{M} we obtain a LD lower bound with rate functional that is $\tilde{\mathcal I}^*_{\textrm{lse}}$, the lower semicontinuous envelope of $\tilde{\mathcal I}^*$. The lower bound is then obtained by showing that $\tilde{\mathcal I}^*_{\textrm{lse}}=\mathcal I^*$.

	\paragraph{\bf Upper bound}
	The upper bound is obtained by the classic argument; by the exponential tightness we can consider compact subsets $\mathcal K$. We have the following estimates. Let $\gamma^I, \gamma^B$ be positive $x$ dependent probability measures, then we have the following classic steps
	\begin{eqnarray*}
		& &\mathbb P^{\mathcal S,*}_N\left((\hat z_N, \hat{\Pi}^B_N, \hat{\Pi}^I_N)\in \mathcal K\right)
		=\mathbb E_{\mathbb P^{\gamma^I,\gamma^B}_N}\left[\frac{d \mathbb P^{\mathcal S,*}_N}{d \mathbb P^{\gamma^I,\gamma^B}_N}\id\left((\hat z_N, \hat{\Pi}^B_N, \hat{\Pi}^I_N)\in \mathcal K\right)\right] \\
		&\leq& \exp\left\{-\inf_{\left\{(\pi, \Pi^B, \Pi^I)\in \mathcal K\right\}}\sum_{a,b}\int_{[0,1]}\left[\Pi^I(a,b; dx)\log\frac{ \gamma^I_{a,b}(x)}{\mu^I_{a,b}}+ \Pi^B(a,b; dx)\log\frac{ \gamma^B_{a,b}(x)}{\mu^B_{a,b}}\right]\right\}\,.
	\end{eqnarray*}
	Since the above estimate is true for any positive $\gamma^B, \gamma^I$ we can optimize getting
	\begin{eqnarray*}
		& &	\limsup_{N\to +\infty} \frac 1N \log \mathbb P^{\mathcal S,*}_N\left((\hat z_N, \hat{\Pi}^B_N, \hat{\Pi}^I_N)\in \mathcal K\right)\leq \\
		& -&	\sup_{\left\{\gamma^B,\gamma^I\right\}}\inf_{\left\{(\pi, \Pi^B, \Pi^I)\in \mathcal K\right\}}\sum_{a,b}\int_{[0,1]}\left[\Pi^I(a,b; dx)\log\frac{ \gamma^I_{a,b}(x)}{\mu^I_{a,b}}+ \Pi^B(a,b; dx)\log\frac{ \gamma^B_{a,b}(x)}{\mu^B_{a,b}}\right] \\
		&= &-\inf_{\left\{(\pi, \Pi^B, \Pi^I)\in \mathcal K\right\}}\sup_{\left\{\gamma^B,\gamma^I\right\}}\sum_{a,b}\int_{[0,1]}\left[\Pi^I(a,b; dx)\log\frac{ \gamma^I_{a,b}(x)}{\mu^I_{a,b}}+ \Pi^B(a,b; dx)\log\frac{ \gamma^B_{a,b}(x)}{\mu^B_{a,b}}\right] \\
		&=& -\inf_{\left\{(\pi, \Pi^B, \Pi^I)\in \mathcal K\right\}}\mathcal I^*\Big(z,\Pi^B,\Pi^I\Big)\,,
	\end{eqnarray*}
	and this is the LD upper bound.
	In the first inequality we optimized over $\gamma^B \gamma^I$; in the second equality
	we used a minmax lemma (see for example lemmas 3.2 and 3.3 in Appendix 2 of \cite{KL}) to exchange infimum and supremum, and the last equality follows from a direct computation by using Lagrange multipliers.
	 	\smallskip

	\subsection{The contraction}
	
	By contraction we can deduce the LD rate functional for the empirical measure of the TASEP configuration \eqref{imp}. We proceed by successive minimizations. We first contract the rate functional $\mathcal I^{\mathfrak S}(z,\Pi^B,\Pi^I)$ to the rate functional $\mathfrak I^{\mathfrak S}(z,\Pi)$ for the pair $\left(\hat z_N,\hat{\Pi}_N\right)$. By contarction we have 
	\begin{equation}\label{contr}
		\mathfrak I^{\mathfrak S}(z,\Pi)=\inf_{\left\{\Pi^B+\Pi^I=\Pi\right\}} \mathcal I^{\mathfrak S}(z,\Pi^B,\Pi^I)\,.
	\end{equation}
	We deduce the following result.
	\begin{proposition}\label{nonholeta}
		When the variables $(\eta,\zeta)$ are distributed according to the Markov bridge \eqref{bridgecompact}, then the pair $\left(\hat z_N,\hat{\Pi}_N\right)$  satisfies a LD principle with a rate functional that is $+\infty$ if at least one among the constraints A, B, C and D is violated and it is otherwise given by
		\begin{equation}\label{mirt}
			\mathfrak I^{\mathfrak S}\Big(z,\Pi\Big)=z(0)\log\frac{\alpha}{1-\alpha}+z(1)\log\frac{\beta}{1-\beta}+\int_0^1 h\left(\Pi(x)|\mu^I\right)\, dx\ +C\,,
		\end{equation}
		where $C$ is the same constant in Lemma \ref{etre}.
	\end{proposition}
	\begin{proof}
		Since $z$ is fixed, to perform the minimization in \eqref{contr}, we need to minimize, for each $x$ such that $z(x)=0$, the following expression
		\begin{equation}\label{stranger}
			\Big[m(x)h\left(\nu^B(x)|\mu^B\right)+(1-m(x))h\left(\nu^I(x)|\mu^I\right)\Big]+m(x)\nu^B_{0,0}(x)\log 2\,,
		\end{equation}
		under the condition
		\begin{equation}\label{vincolo1}
			\left\{m(x),\nu^B(x),\nu^I(x)\,:\, m(x)\nu^B(x)+(1-m(x))\nu^I(x)=\Pi(x)\right\}\,
		\end{equation}
		where we recall that $\sum_{a, b}\nu^B_{a, b}(x)=1$ and $\sum_{a, b}\nu^I_{a, b}(x)=1$. When $z(x)>0$ we do not need to do a minimization since $m(x)$ is constrained to be zero.
		
		By writing the variational conditions with the Lagrange multipliers for the variational problem \eqref{stranger}, namely \eqref{vincolo1} with also the normalization of $\nu^B$ and $\nu^I$, one finds that the resulting system of equations has no critical points for any $m(x) \in (0,1)$. The minimizer must be therefore on the boundary, i.e. corresponding to $m(x)=0$ or $m(x)=1$. Note that if $m(x)=0$ it follows that $\Pi(x)=\nu^I(x)$ while if $m(x)=1$ then $\Pi(x)=\nu^B(x)$.
		
		\smallskip
		
		If the probability measure $\Pi(x)$ satisfies $\Pi_{0,-1}(x)>0$ then the minimum is necessarily attained at $m(x)=0$. Indeed when $m(x)=1$ we have $\Pi(x)=\nu^B(x)$ and we have $\nu^B_{0,-1}(x)=0$ for any $\nu^B$ ; hence the value $m(x)=1$ is not admissible when $\Pi_{0,-1}(x)>0$.
		
		When $\Pi_{0,-1}(x)=0$ we obtain the same value if we compute \eqref{stranger} for $m(x)=0$ and $m(x)=1$; therefore, also in this case, we can consider $m(x)=0$. We conclude that the large deviations rate functional $\mathfrak I^{\mathfrak S}$ is obtained having as a minimizer in \eqref{contr} $\Pi^B=0$ and $\Pi^I=\Pi$ obtaining \eqref{mirt} as a result.
	\end{proof}
	
		In Proposition \ref{nonholeta} the constant $C$ that appears in the large deviation rate functional is fixed by the normalization factor of the Markov bridge $\mathbb P^{\mathfrak S, f,g}_{N}$, i.e. by $\mathfrak{Z}_N$; we recall that this constant $C$ coincides with the one in Lemma \ref{etre}. By construction, $\mathfrak{Z}_N$ is given by $\mathfrak{Z}_N=\frac{Z_N}{4^N}$ with $Z_N=\bra{y}(M^{(0)}+M^{(1)})^N \ket{x}$. More precisely, $C$ is determined by the asymptotic behaviour of $Z_N$ (see, \cite{Der3}), up to the additive constant $-\log4$. In particular, we obtain that $C:=C(\alpha,\beta)=-\log4-\log(\bar{\rho}(1-\bar{\rho}))$, where $\bar{\rho}$ is the limiting particle density:
		\begin{align*}
			\bar{\rho}= \begin{cases}
				\frac{1}{2} &  \alpha \ge 1, \ \ \beta \ge 1 \\
				\alpha & \alpha > \frac{1}{2}, \ \ \beta>\alpha \\ 
				1-\beta & \beta > \frac{1}{2} , \ \ \alpha > \beta.
			\end{cases}
		\end{align*}
	
	\smallskip
	Our next step is to deduce by contraction the LD rate functional $I^{\mathfrak S}(z,\rho)$ for the pair $(\hat z_N, \hat \pi_N)$ from \eqref{mirt}. Let us introduce the function $H(x):= x \log x+(1-x)\log(1-x)$.
	\begin{proposition}
		When the variables $(\eta,\zeta)$ are distributed according to the Markov bridge \eqref{bridgecompact}, then the pair $\left(\hat z_N,\hat{\pi}_N\right)$  satisfies a LD principle with rate functional $I^{\mathfrak S}(z,\rho)$ that is $+\infty$ if $\rho$ and $z$ are not absolutely continuous or if condition C is violated. Moreover, the rate functional is also $+\infty$ if for almost all $x$ the densities $(\dot z(x), \rho(x))$ do not belong to the region $R\subset \mathbb R^2$ defined by 
		\begin{equation}\label{R}
			R=\left\{(\dot z, \rho)\,:\, 0\leq \rho \leq 1\,, 0\leq \rho - \dot z \leq 1\right\}.
		\end{equation}
		In the remaining cases the rate functional is finite and it is given by
		\begin{equation}\label{rhoezeta}
			I^{\mathfrak S}(z,\rho)=z(0)\log\frac{\alpha}{1-\alpha}+z(1)\log\frac{\beta}{1-\beta}+\int_0^1 \Big[H\big(\rho(x)\big)+H\Big(\rho(x)-\dot{z}(x)\Big)\Big] \, dx +  C'
		\end{equation}
		where $C'$ is a normalization constant.
	\end{proposition}
	\begin{proof}
		By the contraction principle, the proof follows by a direct minimization of \eqref{mirt} over $\Pi$ with the constraints
		\begin{equation}
			\left\{
			\begin{array}{l}
				\rho(x)=\Pi_{1,1}(x)+\Pi_{1,0}(x)\\
				\dot z(x)=\Pi_{1,1}(x)-\Pi_{0,-1}(x)\,.
			\end{array}
			\right.
		\end{equation}
		The geometric constraints \eqref{R} follow by the above relation and the fact that $\Pi$ is a probability measure. 
	\end{proof}
	
	A direct computation shows that the constant in \eqref{rhoezeta} is given by $C':=C'(\alpha,\beta)=C(\alpha,\beta)+\log4$ where $C(\alpha,\beta)$ is the constant in \eqref{mirt}.
		
		\smallskip
	
	We can now obtain the LD rate function $I^{\mathrm{TASEP}}(\rho)$ for the empirical measure \eqref{imp} when $\eta$ is distributed according to the invariant measure of boundary driven TASEP by contraction from \eqref{rhoezeta}, i.e. we have $I^{\mathrm{TASEP}}(\rho)=\inf_z 	I^{\mathfrak S}(z,\rho)$. 
	We define: 
	\begin{equation*}
		a=\frac{1-\alpha}{\alpha}, \ \ \ b=\frac{1-\beta}{\beta}\,
	\end{equation*}
	with this notation we are  considering the regime of TASEP corresponding to $ab<1$.
	This last contraction gives as a result the rate functional in \cite{Der4, Derr7}. Our representation is obtained as an infimum and we write the final expression in the same form of \cite{Bryc} where its equivalence with the representation in \cite{Der4, Derr7} is shown.
	
	\begin{theorem}
		Let $F(x)=\int_0^x\rho(y)dy$ and $\mathcal G$ the set of absolutely continuous functions $G$ such that $G(0)=0$ and $0\leq \dot G(x)\leq 1$. Let $\eta$ be distributed according to the invariant measure of the boundary driven TASEP when the boundary parameters are such that $ab<1$. We have that the large deviations rate functional for $\hat\pi_N$ in \eqref{imp} is given by
		\begin{equation*}
			\begin{aligned}
				I^{\mathrm{TASEP}}(\rho)
				= \inf_{G\in \mathcal G} \biggl\{
				&\int_0^1 \Bigl[H\left(\dot F(y)\right) + H\left(\dot{G}(y)\right)\Bigr]\, dy  \\
				&\quad + \ln(ab)\, \min_{x\in[0,1]} \bigl[F(x)-G(x)\bigr] - \ln(b)\, \bigl(F(1)-G(1)\bigr)
				\biggr\}.
			\end{aligned}
		\end{equation*}
	\end{theorem}
	\begin{proof}
		By the contraction principle  we have to  minimize the functional \eqref{rhoezeta} over the functions $z(x) \ge 0$ such that $0 \le \rho(x)-\dot z(x) \le 1$ for almost any $x\in[0,1]$.
		
		We define:
		\begin{align*}
			& F(x)= \int_0^x \rho(y)\,dy \\
			&G(x)=F(x)-\Big(z(x)-z(0)\Big)\,,
		\end{align*}
		and note that $G(0)=0$ and $\dot G(x)=\rho(x)-\dot z(x)$ so that $G\in \mathcal G$; moreover $F(x)+z(0)-G(x)=z(x)\geq 0$.
		The functional \eqref{rhoezeta} can be written in terms of $F, G, z(0)$ as
		$$
		\int_0^1 \Bigl[H\left(\dot F(y)\right) + H\left(\dot{G}(y)\right)\Bigr]\, dy  
		-z(0)\ln(ab)-\bigl[F(1)-G(1)\bigr]\ln(b):=\mathcal L(F, G, z(0))\,.
		$$
		By the contraction principle we have
		We rewrite the problem:
		\begin{equation*}
			I^{\mathrm{TASEP}}(\rho)=\inf_{\big\{z(0) \ge 0\big\}}  \inf_{\big\{G\in \mathcal G : G(x) \le F(x)+z(0)\big\}} \mathcal{L}(F, G, z(0)) \,.
		\end{equation*}
		We exchange the order of the infima obtaining that the above variational problem is equivalent to
		\begin{equation*}
			I^{\mathrm{TASEP}}(\rho)=\inf_{\big\{G\in \mathcal G\big\}}  \inf_{\big\{z(0)\ge \max_{x \in [0,1]}[G(x)-F(x)]\big\}} \mathcal{L}(F, G, z(0)) \,.
		\end{equation*}
		Since $ab<1$ the value of $z(0)$ that minimizes the first infimum is exactly $z(0)= \max_{x \in [0,1]}[G(x)-F(x)]$ and this finishes the proof.
	\end{proof}
	
	\subsection*{Declarations}
	The authors have no conflicts of interest. Data sharing is not applicable to this article as no datasets were generated or analysed during the current study.
	
		\subsection*{Acknowledgements}
	We thank M. Goldwurm and F. Mignosi for useful discussions.
	DG acknowledges the financial support from the Italian Research Funding Agency (MIUR) through
	PRIN project ``Emergence of condensation-like phenomena in interacting particle systems: kinetic and lattice models'', grant n. 202277WX43.

	\appendix
	
	\section{Some mathematical tools}

	We recall here some classic result and deduce some facts that we use in the paper. 
	
	\subsection{Finite Perron-Frobenius Theorem}
	\label{PF-finite}
	Here we briefly recall the classic Perron Frobenius Theorem with also some auxiliary lemmas. We refer to \cite{K} for a complete discussion. The results hold true for any finite alphabeth.

	\begin{theorem}[Perron-Frobenius]\label{thPFf}
		Consider a finite $B\times B$ irreducible and aperiodic non-negative matrix $M$, then there exists a positive maximal eigenvalue $\lambda>0$ such that all the remaining eigenvalues $\lambda_i$ are such that $|\lambda_i|<\lambda$. The unique eigenvector $\ket e=(e(b))_{b\in B}$ corresponding to the maximal eigenvalue can be fixed in such a way that $e(b)>0$ for any $b\in B$.
	\end{theorem}

	\begin{remark}
		When the matrix $M$ is irreducible but not aperiodic, Theorem \ref{thPFf} has to be modified and there exists still a maximal real eigenvalue $\lambda$ and a corresponding strictly positive eigenvector $\ket e$; however, for the other eigenvalues we do not have a strict inequality, but instead $|\lambda_i|\leq\lambda$. If the period of the irreducible chain is $p$ then there are exactly $p$ eigenvalues of maximal modulus $\lambda$. We refer to \cite{K} for more details.
	\end{remark}
	
	As a consequence of the Perron Frobenius Theorem we have the following generalized Doob transform.
	\begin{lemma}\label{stoclemma}
		Let $M$ be a non-negative irreducible and aperiodic matrix and let $\lambda$, $\ket e$ be the corresponding maximal eigenvalue and positive right eigenvector; then the matrix
		\begin{equation}\label{stocmat}
			S_{b,b'}=\frac 1\lambda e(b)^{-1} M_{b,b'}e(b')\, \qquad b,b'\in B\,,
		\end{equation}
		is a stochastic matrix and we have
		\begin{equation}\label{power}
			S^{k}_{b,b'}=\frac {1}{\lambda^k} e(b)^{-1}  \left(M^{k}\right)_{b,b'} e(b') \, \qquad  k\in \mathbb N,.
		\end{equation}
		If $M$ is irreducible and aperiodic, the matrix $S$ is also irreducible and aperiodic.
	\end{lemma}
	\begin{proof}
		The proof of \eqref{stocmat} follows by a direct computation. We need just to check that the sum along each row is one
		$$
		\sum_{b'\in B} S_{b,b'}=\frac{1}{\lambda}  e(b)^{-1}\sum_{b'\in B} M_{b,b'} e(b')=\frac 1\lambda e(b)^{-1}\lambda e(b)=1\,.
		$$
		Formula \eqref{power} follows directly by the representation \eqref{matriceScompact}. Irreducibility and aperiodicity follow directly.
	\end{proof}
	Notice that in formula \eqref{power} the upper label $k$ of the matrices $S,M$ is an integer number, not an element of $A$, and it is denoting the $k$-th power of such matrices.

	\subsection{Other tools}
	A simple but powerful result that we use is the following.
	\begin{lemma}\label{lemmac}
		Suppose that we have two sequences of probability measures $\mu_N$ and $\mu_{N}'$ such that there exist two sequences of positive constants $k_N,K_N$ such that
		\begin{equation}
			0\leq \liminf_{N\to +\infty}\frac 1N\log k_N\leq \limsup_{N\to +\infty} \frac 1N\log K_N\leq 0\,, 
		\end{equation}
		with $k_N\mu_N'\leq \mu_N\leq K_N\mu_N'$. If $\mu_N'$ satisfies a large deviations principle then also $\mu_N$ satisfies a large deviations principle with the same rate.
	\end{lemma}
	\begin{proof}
		The proof follows from a direct verification of the validity of the lower and upper bounds.
	\end{proof}
	
	\subsection{Eigenvalues and eigenvectors for a class of infinite matrices}\label{eigeninf}
	
	We consider the eigenvalue problem for the infinite-dimensional matrices, having rows and columns labeled by $\mathbb N_0$, of the form:  \newline
	\begin{equation}\label{infmat}
		A= \begin{bmatrix}
			\alpha & \beta_1 & 0 & \cdots & 0 & \cdots &\cdots \\ \beta_2 & \alpha &\beta_1 & 0&  \cdots & 0 &\cdots \\ 0 & \beta_2 &  \alpha & \beta_1 & 0 & \cdots \\ 0 & 0 & \beta_2 & \alpha & \beta_1 & 0 & \cdots \\ \vdots & \vdots  &0 & \ddots & \ddots & \ddots & \ddots  \\ \vdots & \vdots & \vdots & \ddots & \ddots &\ddots &\ddots \\ \end{bmatrix}	\end{equation} 
	that is
	\begin{equation*}
		A_{ij}= \begin{cases*} 
			\beta_1 \delta_{i,j+1} \\ \alpha \delta_{i,j} \ \ \ \ \ \ \ \ \ \  i,j=0,1,2,... \\ \beta_2 \delta_{i-1,j}
		\end{cases*}
	\end{equation*}
	with $\alpha, \beta_1, \beta_2 > 0$. The spectral theory in the case of infinite matrices is more complex and we illustrate just a few computations. We refer, for example, to \cite{K} for more details.
	
	The Perron value for the matrix \eqref{infmat} is defined as
	\begin{equation}\label{perron}
		\lim_{n \rightarrow +\infty} \sqrt[n]{\left(A^n\right)_{i,j}}\,,
	\end{equation}
	and the limit does not depend on $i,j$, \cite{K}. These matrices have a continuum of positive eigenvalues larger than the Perron value, for which the corresponding eigenvectors are positive. We have
	
	\begin{lemma}\label{perronlemma}
		The Perron value of the matrix \eqref{infmat} is given by $\alpha+2\sqrt{\beta_1\beta_2}$ and the corresponding column eigenvector is given by 
		$\ket e= \left((n+1) \left(\frac{\beta_2}{\beta_1} \right)^{\frac{n+1}{2}}\right)_{n\in \mathbb N_0}$.
	\end{lemma}
	\begin{proof}
		We give only an outline of the argument, since the computation is equivalent to the classic combinatorial computation of the partition function based on bicoloured Motzkin and Dyck paths \cite{review}.
		
		We consider a discrete time random walk on $\mathbb N_0$, that starts at zero and performs jumps of $0$ or $\pm 1$. We assign the weight $\beta_1$ to jumps of $+1$, the weight $\beta_2$ to jumps of $-1$ and the weight $\alpha$ to jumps of $0$. The weights of paths of even length that start and end at $0$ is given by
		\begin{align*}
			(A^{2n})_{0,0}= \sum_{k=0}^{n} \mathcal{C}_{n-k} \binom{2n}{2k} \alpha^{2k} \beta_1^{n-k} \beta_2^{n-k}
		\end{align*}
		where $\mathcal{C}_n$ is the $n$-th Catalan number. By Stirling formula and classic asymptotic estimates, we deduce the Perron value. The eigenvector can be verified by a direct computation.
	\end{proof}
	In the special case with $\alpha=2$, $\beta_1=\beta_2=1$ the Perron eigenvalue is $\lambda=4$ and the eigenvector associated is $\ket e=(n+1)_{n\in \mathbb N_0}$.

\end{document}